\documentclass[aos,preprint]{imsart}

\RequirePackage[OT1]{fontenc}
\usepackage[bbgreekl]{mathbbol}
\RequirePackage{amsthm,amsmath, amssymb}
\RequirePackage[]{natbib}
\RequirePackage[colorlinks,citecolor=blue,urlcolor=blue]{hyperref}
\usepackage{todonotes}
\usepackage{mathrsfs}
\usepackage{bbm}
\usepackage{hyperref}
\usepackage{enumerate}
\usepackage{graphicx}
\usepackage{float}

\startlocaldefs
\numberwithin{equation}{section}
\theoremstyle{plain}
\newtheorem{thm}{Theorem}[section]
\newtheorem{lemma}[thm]{Lemma}
\newtheorem{definition}[thm]{Definition}

\newtheorem{corollary}[thm]{Corollary}

\endlocaldefs

\newcommand{\ind}{1\hspace{-0,9ex}1}

\DeclareMathOperator{\Var}{Var}
\DeclareMathOperator{\E}{E}

\DeclareMathOperator{\tr}{tr}

\DeclareMathOperator{\rank}{rank}

\def\d{\mathrm{d}}

\def\E{\mathbb{E}}

\def\N{\mathbb{N}}

\def\R{\mathbb{R}}

\def\N{\mathbb{N}}
\def\P{\mathbb{P}}

\newenvironment{theorem}{\begin{thm}}{\end{thm}}

\allowdisplaybreaks

\begin{document}

\begin{frontmatter}

	\title{Weak convergence of the empirical spectral distribution of ultra-high-dimensional banded sample covariance matrices}
	\runtitle{Banded sample covariance matrices}
\thankstext{T1}{Supported by the DFG Research Unit 1735, RO 3766/3-1 and DE 502/26-2.}
	\begin{aug}
		\author{\fnms{Kamil} \snm{Jurczak}\ead[label=e1]{kamil.jurczak@ruhr-uni-bochum.de}}

		\affiliation{Ruhr-Universit\"at Bochum}

		\address{Fakult\"at f\"ur Mathematik \\
		Ruhr-Universit\"at Bochum \\
		44780 Bochum \\
		Germany \\
		\printead{e1}\\
		}
	\end{aug}

\begin{abstract}In this article we investigate high-dimensional banded sample covariance matrices under the regime that the sample size $n$, the dimension $p$ and the bandwidth $d$ tend simultaneously to infinity such that $$n/p\to 0 \ \ \text{and} \ \ 2d/n\to y>0.$$ It is shown that the empirical spectral distribution of those matrices almost surely converges weakly to some deterministic probability measure which is characterized by its moments. Certain restricted compositions of natural numbers play a crucial role in the evaluation of the expected moments of the empirical spectral distribution. 
\end{abstract}

\end{frontmatter}

\small

\section{Introduction}
In statistics, high-dimensional sparse sample covariance matrices naturally occur as regularized estimators of population covariance matrices in high dimensions provided most entries are known to be zero or close to zero, cf. \cite{Bickel2008b}, \cite{Levina2012}. Statistical properties of these type of estimators had been intensively studied in recent years. Let us just mention some few crucial contributions. \cite{Karoui2008} provided a consistent estimate under the spectral norm for certain sparse sample covariance matrices based on thresholding. \cite{Lam2009} studied the rate of convergence of estimators  for sparse covariance matrices and precision matrices based on penalized likelihood. \cite{Caib} determined the minimax rates for sparse covariance matrix estimation under various matrix norm losses over appropriate classes of covariance matrices. As a special case of sparse covariance matrices arise banded covariance matrices. For the latter, it is a priori known  that the non-zero entries do not lie too far from the diagonal. \cite{Bickel2008} investigated a regularized estimator for banded covariance matrices and its rate of convergence. \cite{Qiu2012} proposed a test for bandedness.\\   
Apart from statistics, sparse sample covariance matrices are applicable in models of physical systems, where most particles do not interact with each other, see \cite{Bai2007}. Despite this rich occurrence, there is not much known about the spectral properties of high-dimensional sparse sample covariance matrices as compared to the classical high-dimensional sample covariance matrices. Under some slight regularity assumptions \cite{Bai2007} have proved that the empirical spectral distribution of
\begin{align*}
\frac{1}{\sqrt{nd}}(\boldsymbol X_n^{~}\boldsymbol X_n'-\sigma^2n\boldsymbol 1_{p\times p})\circ \boldsymbol D_p \in \R^{p\times p}
\end{align*}
converges to the semicircular law as $d/n\to 0$ and $p,d,n\to \infty$, where 
the entries of $\boldsymbol X_n=(X_{ik,n})_{ik}\in\R^{p\times n}$ are independent, centered random variables with variance $\sigma^2>0$, the symmetric matrix $\boldsymbol D_p=(D_{ij,p})\in\R^{p\times p}$ is independent of $\boldsymbol X_n$ with $\sum_{i=1}^p \E|D_{ij,p}|^2=d+o(d)$, and $\circ$ denotes the Hadamard product. In particular, the case, that $\boldsymbol D_p$ is a deterministic $0$-$1$-sparsity mask with $d$ non-zero entries per column, is covered by this model. The assumption $d/n\to 0$ is crucial for their result. On the contrary, an intrinsic consequence of the investigation in this article is that for $d/n\to y>0$ the limiting spectral distribution of a sparse sample covariance - if existent - does essentially depend on the structure of the sparsity mask through the number of certain restricted compositions of natural numbers. However, the focus in this article lies on the special case of banded sample covariance matrices. For those we prove that their sequence of empirical spectral distributions almost surely converges weakly, where the limiting distribution is described by its moments. \\
In contrast to banded or sparse sample covariance matrices, Wigner matrices with an additional sparsity structure have been extensively studied. Let us just mention a few contributions.  \cite{Bogachev1991} proved under slight regularity conditions that the empirical spectral distribution of sparse Wigner matrices converges weakly to the semi-circular law. \cite{Benaych2013} showed that its largest eigenvalue converges to $2$ in probability and that eigenvectors corresponding to eigenvalues far enough from zero are delocalized if the number of non-zero entries per row is of larger order than $(\log N)^{6(1+\alpha)}$, where $N$ is the number rows of the random matrix and the parameter $\alpha$ depends on the tails of the underlying distribution. Further, \cite{Benaych2014} studied localization and delocalization of eigenvectors for heavy-tailed band Wigner matrices. In an extraordinary article \cite{Sodin2010} investigated the limiting distribution of the smallest and largest eigenvalues of band Wigner matrices.\\
The article is structured as follows. In the rest of this section we introduce the basic notation, recall some useful results,  and summarize the method of moments. In Section \ref{section: combinatorial} we compile some combinatorial tools to evaluate the expected moments of the spectral distribution of banded sample covariance matrices. The concept of ordered trees with a $d$-band structure on the $I$-line is introduced, and an expansion for the number of these so-called $d$-banded ordered trees with a fixed number of vertices is given by means of restricted compositions of natural numbers. Finally, Section \ref{section: main} is devoted to the main result concerning the almost sure weak convergence of the spectral distribution of banded sample covariance matrices and its proof.

\subsection{Preliminaries}
We denote the ordered eigenvalues of a symmetric matrix \linebreak $\boldsymbol A\in\R^{p\times p}$ by $\lambda_1(\boldsymbol A)\ge\dots\ge\lambda_p(\boldsymbol A)$. Then, the spectral distribution of $\boldsymbol A$ is the normalized counting measure on the eigenvalues of $\boldsymbol A$  
$$\mu^{\boldsymbol A}:=p^{-1}\sum_{i=1}^p\delta_{\lambda_i(\boldsymbol A)},$$
where $\delta_x$ is the Dirac measure on $x$. We write
\begin{align*}
d_L(\mu,\nu)=\inf \Big\{\varepsilon>0 \ &\Big\arrowvert \ \mu((-\infty,x-\varepsilon])-\varepsilon\le \nu((-\infty,x])\\
&\hspace{1cm} \le \mu((-\infty,x+\varepsilon])+\varepsilon  \text{ for all }x\in\R\Big\}
\end{align*}
for the L\'evy distance between two probability measures $\mu$ and $\nu$. Moreover, we will also use frequently the Kolmogorov distance
$$\d_K(\mu,\nu)=\sup_{x\in\R}|\mu((-\infty,x])-\nu((-\infty,x])|.$$
Recall the basic relation $d_L(\mu,\nu)\le d_K(\mu,\nu)$.
We abbreviate the set $\{1,\dots,p\},~p\in\N,$ by $[p]$. For any subset $N\subset [p]^2$ the matrix $\boldsymbol A=\boldsymbol 1_N\in\R^{p\times p}$ has entries $A_{ij}=\ind_N(i,j)$, where $\ind_N:[p]^2\rightarrow \{0,1\}$ is the indicator function on $N$. For $N:=\{(i,j)\in [p]^2: |i-j|\le d\}$ we define $\boldsymbol1_{d}:=\boldsymbol1_{N}$.
For an expression $f(p,d,n,l)$ we write $O_l(g(p,d,n,l))$ if there exists a positive function $h$ such that $f(p,d,n,l)\le h(l)g(p,d,n,l)$ for all $p,d,n,l$.\\ 
Let us recall some useful results to bound the L\'evy distance between the spectral distributions of two symmetric matrices $A,B\in\R^{p\times p}$.

\begin{theorem}[Theorem A.43 of \cite{Bai2010}]\label{theorem: A43}
Let $A$ and $B$ be two $p\times p$ symmetric matrices. Then,
\begin{equation}
d_K\left(\mu^A,\mu^B\right)\le \frac{1}{p}\rank(A-B),
\end{equation}
where $\mu^A$ and $\mu^B$ denote the spectral distributions of $A$ and $B$, respectively. 
\end{theorem}

\begin{theorem}[Theorem A. 38 of \cite{Bai2010}]\label{theorem: A38}
Let $\lambda_1,\dots,\lambda_p$ and $\delta_1,\dots,\delta_p$ be two families of real numbers and their empirical distributions be denoted by $\mu$ and $\bar \mu$. Then, for any $\alpha>0$, we have 
\begin{equation}
d_L^{\alpha+1}(\mu,\bar \mu)\le \min_{\pi}\frac{1}{p}\sum_{k=1}^p |\lambda_k-\delta_{\pi(k)}|^\alpha,
\end{equation}
where the minimum is running over all permutations $\pi$ on $\{1,\dots, d\}$.
\end{theorem}

\begin{corollary}[Corollary A.41 of \cite{Bai2010}]\label{corollary: A41}
Let $A$ and $B$ be two $d\times d$ Hermitian matrices with spectral distribution $\mu^A$ and $\mu^B$. Then,
\begin{equation}
d_L^3\left(\mu^A,\mu^B\right)\le \frac{1}{d}\tr\big((A-B)(A-B)^\ast \big).
\end{equation}
\end{corollary}

\subsection{Method of moments}
The method of moments is a tool to deduce weak convergence of a sequence of measures and goes back to \cite{Tchebycheff1890}. \cite{Wigner1958} was the first to apply this technique in random matrices for the purpose of establishing the weak convergence of a sequence of expected empirical spectral distributions 
of Wigner matrices to the semi-circular law. The foundation of the method of moments is the following statement.
\begin{theorem}[Moment convergence theorem]
Let $\mu_n, n\in\N,$ be probability measures on the real line with finite moments $m_{n,r}:=\int x^r\d\mu_n(x),~r\in\N$. Suppose that
$m_r=\lim_{n\to\infty}m_{n,r}$
exists for every $r\in \N$. Then, there exists a probability measure $\mu$ with moments $m_r,~r\in\N$. Moreover, if $\mu$ is the unique probability measure with moments $m_r,~r\in\N,$ then the sequence $(\mu_n)$ converges weakly to $\mu$.  
\end{theorem}
\begin{proof}
For arbitrary $\varepsilon>0$ let $N\in\N$ be sufficiently large such that $$|m_{n,2}-m_2|\le \varepsilon \ \ \text{for all }n\ge N.$$ Then by Markov's inequality,
\begin{align*}
\mu_n([-x,x]^c)\le  \frac{m_{n,2}}{2x^2}\le \frac{m_{2}+\varepsilon}{2x^2}
\end{align*}
for any $x>0$ which implies that the sequence $(\mu_n)$ is tight. 
Hence, by Prokhorov's theorem for any subsequence $(\mu_{n_k})_k$ there exists a subsubsequence $\mu_{n_{k_l}}$ converging weakly to a probability measure $\mu$. We show that the moments of $\mu$ are given by the sequence $(m_r),~r\in\N$. Thereto, let $X_{n_{k_l}}\sim \mu_{n_{k_l}}$ and $X\sim \mu$. By the convergence of all moments, the sequences $(X_n^r)_n$ are uniformly integrable and therefore $X^r$ is integrable, and $\E X_{n_{k_l}}^r\rightarrow \E X^r=m_r$ for all $r\in\N$. Now suppose that $\mu$ is the unique measure with moments $m_r,~r\in\N$. Then each subsequence $(\mu_{n_{r}})$ has a weakly convergent subsubsequence $(\mu_{n_{k_l}})$ with limit $\mu$. This implies the weak convergence of $(\mu_n)$ to $\mu$.
\end{proof}
The question, whether a sequence of moments $m_r,~r\in\N,$ uniquely determines a measure $\mu$ on the real line, is partially answered by Carleman's condition which says that $\mu$ is the only measure with moments $m_r,~r\in\N,$ if
$$\sum_{r=1}^\infty {m_{2r}^{-1/(2r)}}=\infty.$$
This condition is satisfied if the moments do not grow too fast. In particular, thereby all probabilty measures with sub-exponential tails are determined by their moments. In the main theorem of this article the limiting spectral measure has even finite support.

\section{Combinatorial tools}\label{section: combinatorial}
In this section we introduce some basic combinatorial objects which will be useful to prove the convergence of the expected moments of the empirical spectral distribution of a banded sample covariance matrix.

\subsection{Walks on ordered trees}

A (finite simple) \textit{graph} $G=(V,E)$ is a pair of a finite vertex set $V\neq \emptyset$ and an edge set $E\subset\{e\subset V: |e|=2\}$ such that $V\cap E=\emptyset$. 
$\tilde G=(\tilde V,\tilde E)$ is a subgraph of $G$ if $\tilde V\subset V$ and $\tilde E\subset E$.

\noindent
Let $v\in V$ be the vertex of a graph $G=(V,E)$. The vertex $v$ is called \textit{incident} with an edge $e\in E$ if $v\in e$. The number $\deg(v)$ of edges incident with $v$ is the degree of $v$. A vertex $v'$ is said to be a \textit{neighbor} of $v$ if $\{v,v'\}\in E$. If $V=V_1+V_2$ such that for any edge $e\in E$ holds $e\cap V_1=e\cap V_2=1$, then $G$ is a \textit{bipartite graph} with parts $V_1$ and $V_2$. 

\noindent A \textit{walk} of \textit{length} $n-1$ on a graph  $G=(V,E)$ is a sequence of vertices $v_1,\dots,v_n$, $n\in \N$, such that $\{v_i,v_{i+1}\}\in E$ for all $i\in[n-1]$. The vertex $v_1$ is the \textit{start vertex}, $v_n$ the \textit{end vertex} and $v_2,\dots,v_{n-1}$ the \textit{inner vertices}. We say a vertex $v$ is \textit{visited} by a walk $v_1,\dots,v_n$ if $v\in\{v_1,\dots,v_n\}$. 
An edge $e\in E$ is \textit{crossed} by a walk $v_1,\dots,v_n$ if $e=\{v_i,v_{i+1}\}$ for some $i\in[n-1]$. We say the path $v_1,\dots,v_n$ crosses (resp. visits) an edge $e\in E$ (resp. a vertex $v\in V$) at step $k$ if $e=\{v_k,v_{k+1}\}$ (resp. $v=v_{k+1}$). 
A walk is \textit{closed} if the start vertex and the end vertex coincide. Further, a walk $v_1,\dots,v_n$ on a graph $G$ visiting each edge at most once is a \textit{path}. A \textit{circle} is closed walk $v_1,\dots,v_n,v_1$, where $v_1,\dots,v_n$ is a path.

\noindent A graph $G=(V,E)$ is said to be connected if for any pair of vertices $v,v'\in V$ there exists a path on $G$ from the start vertex $v$ to the end vertex $v'$. A (connected) component of $G$ is a graph $\tilde G$ with $\tilde V\subset V$ and $\tilde E=E\cap\{e\subset\tilde V~:~|e|=2\}$ such that $\tilde G$ is connected and for any two vertices $\tilde v \in \tilde V$ and $v\in V\setminus \tilde V$ there does not exist a path from $\tilde v$ to $v$ in $G$. In other words, the components of a graph $G$ are the maximal connected subgraphs of $G$. 

\noindent
A connected graph $G=(V,E)$ is called a tree if for any edge $e\in E$ the graph $(V,E\setminus\{e\})$ is not connected. That is, there is exactly one path from  $v\in V$ to $v'\in V$ for any two vertices $v,v'\in V$, and trees are free of circles. Moreover, it is well-known that a connected graph on $n\in\N$ vertices is a tree if and only if it contains exactly $n-1$ edges.

\noindent
A \textit{rooted} tree is a pair $(G,v_{root})$, where $G=(V,E)$ is a tree and $v_{root}\in V$ is a designated vertex of $G$ called the \textit{root} of $G$. There is a natural partial ordering on a rooted tree. We write $v\le_G v'$ if $v$ is visited by the path with start vertex $v_{root}$ and end vertex $v'$. Clearly, the root satisfies $v_{root}\le_G v$ for all $v\in V$. A neighbor $v'\in V$ of a vertex $v\in V$ with $v\le_G v'$ is a \textit{child} of $v$ and $v$ is the \textit{parent} of $v'$. 
\begin{figure}[htbp] 
  \centering
     \includegraphics[width=0.5\textwidth]{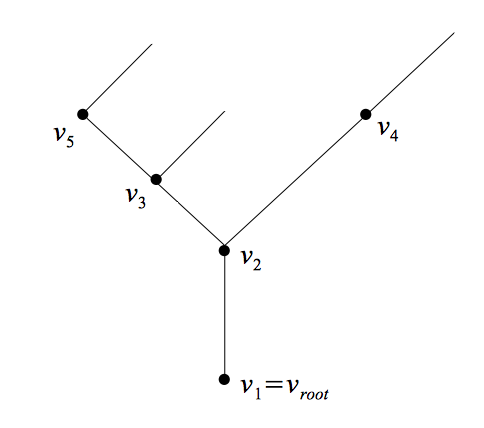}
  \caption{Example of a rooted tree $G$ with $v_1\le_G v_2\le_G v_3 \le_G v_5$ and $v_1\le_G v_2\le_G v_4$.}
  \label{fig:Bild1}
  \end{figure}

\noindent
A vertex $v\neq v_{root}$ has $\deg(v)-1$ children and one parent. If the children of each vertex $v\in V$ are equipped with a total order $\le_v$ then $G$ is called an ordered tree or plane tree. The last name is justified because there is a natural embedding of the graph into the plane by drawing the children of a vertex increasing from left to right. 
\begin{figure}[htbp] 
  \centering
     \includegraphics[width=0.9\textwidth]{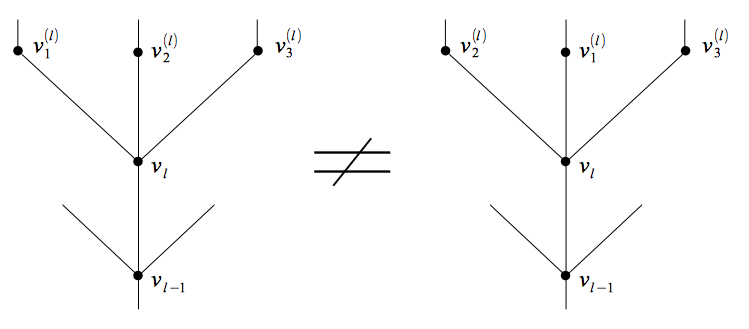}
  \caption{Example of two different ordered trees: The order on the children of $v_l$ in the left graph is $v_1^{(l)}\le v_2^{(l)}\le v_3^{(l)}$, whereas the depiction of the right graph implies $v_2^{(l)}\le v_1^{(l)}\le v_3^{(l)}$.}
  \label{fig:Bild2}
  \end{figure}
An ordered tree may be associated with a closed walk $v_1,\dots,v_{2|V|-1}$ ``around the tree'' defined by the following inductive procedure. The root $v_{root}$ is the starting vertex $v_1$ of the walk. Let $v_{k}\in V$ be the vertex visited at the $(k-1)$-th step and $v^{(k)}_1,\dots,v^{(k)}_l\in V$ its children. If the walk has already visited the vertex $v^{(k)}_i$ for some $i<l$ but not the vertex $v^{(k)}_{i+1}$ then $v_{k+1}=v^{(k)}_{i+1}$, otherwise $v_{k+1}$ is the parent of $v_k$ respectively $v_k$ is the end vertex of the walk if $v_{k}=v_{root}$. It is easy to see that the walk crosses all edges of the tree once in each direction and hence the procedure stops right after $2|V|-2$ steps at the root. On the other hand, let $v_1,\dots,v_{2n-1}$ be a closed walk on $G=(V,E)$ with $V=\{v_1,\dots,v_{2n-1}\}$, $|V|=n,$ and $E=\{\{v_i,v_{i+1}\}~:~i\in[n-1]\}$ such that each edge $e\in E$ is crossed at least twice by the walk $v_1,\dots,v_{2n-1}$. Since $|E|\le |V|-1$ and $G$ is connected it holds $|E|=|V|-1$. Hence, $G$ is a tree. Let $v_{root}:=v_1$ be the root of $G$. Then, for each vertex $v\in V$ of the tree there is a natural order on its children induced by the increasing sequence in which they have been visited by the walk for the first time. On a fixed vertex set $V=\{v_1,\dots,v_n\}$ this defines a bijection between ordered trees on $V$ and closed walks in $V$ crossing an each edge at least twice. Subsequently for an ordered tree $G$ this walk is called the canonical walk on $G$.
\begin{figure}[htbp] 
  \centering
     \includegraphics[width=0.7\textwidth]{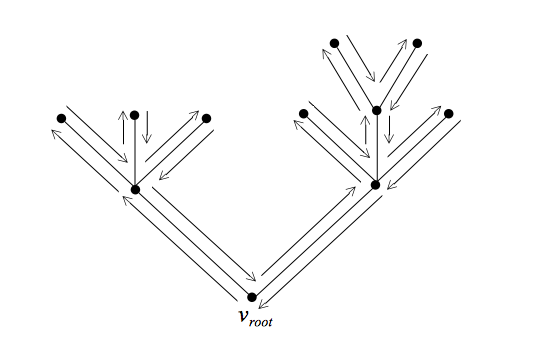}
  \caption{Example of a canonical walk on an ordered tree: The walk starts in the root and runs clock-wisely.} 
  \label{fig:Bild3}
  \end{figure}\\
  Two ordered trees $G$ and $G'$ with vertex set $V$ and $V'$ are isomorphic if there exists a bijection $\pi:V\rightarrow V'$ such that $v_1,\dots,v_{2|V|-1}$ is the canonical walk on $G$ and $\pi(v_1),\dots,\pi(v_{2|V|-1})$ the canonical walk on $G'$. The mapping $\pi$ is called an isomorphism. Let $\pi$ be an isomorphism from $G$ to $G'$, then the following properties are satisfied:
\begin{enumerate}
\item $G$ and $G'$ have the same number of vertices and edges.
\item Let $v,w\in V$. Then, $\{v,w\}$ is an edge of $G$ if and only if $\{\pi(v),\pi(w)\}$ is an edge of $G'$.
\item $v_1$ is the root $G$ and $\pi(v_1)$ the root of $G'$.
\item Let $v,w$ be two vertices of $G$. Then, $v\le_G w$ on $G$ if and only if $\pi(v)\le_G \pi(w)$ on $G'$.
\item Let $v$ be be a vertex of $G$, and $w_1$ and $w_2$ two of its children. Then, $w_1\le_v w_2$ if and only $\pi(w_1)\le_{\pi(v)}\pi(w_2)$. 
\end{enumerate} 
Each tree $G=(V,E)$ is a bipartite graph. To see this, fix some vertex $v\in V$. Then, define $V_1:=\{v'\in V:  \text{The path }v,\dots,v' \text{ has even length}\}$ and  $V_2:=\{v'\in V:  \text{The path }v,\dots,v' \text{ has uneven length}\}$. The sets $V_1$ and $V_2$ are well-defined since on a tree there is exactly one path with start vertex $v$ and end vertex $v'$. If $v',v''\in V_1$ or $v',v''\in V_2$ then it holds $\{v',v''\}\not\in E$, since for any two vertices $v',v''\in V$ with $\{v',v''\}\in E$ the length of the paths $v,\dots,v'$ and $v,\dots,v''$ differs by $1$. So, either $v,\dots,v'$ or $v,\dots,v''$ has even length, and therefore $V=V_1+V_2$. 
\begin{figure}[htbp] 
  \centering
     \includegraphics[width=0.6\textwidth]{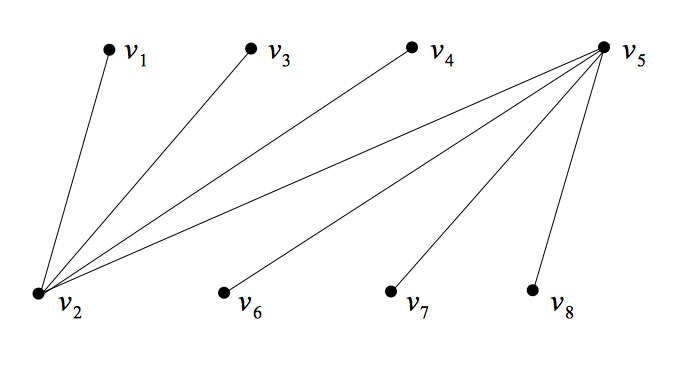}
  \caption{Depiction of a tree via its two parts.} 
  \label{fig:Bild6}
  \end{figure}\\
Let $p,n\in \N$ and refer to the set $\{(i,1)~:~i\in [p]\}$ as the $I$-\textit{line} and $\{(k,0)~:~k\in [n]\}$ as the $K$-\textit{line}. Subsequently we will only consider ordered trees $G=(V_1+V_2,E)$ such that the part $V_1$ containing the root of $G$ is a subset of $I$-line, whereas $V_2$ is subset of the $K$-line. We will usually identify the elements on the $I$-line and on the $K$-line with its first component where a label $i$ always refers to a vertex on the $I$-line and $k$ to a vertex on the $K$-line. Moreover, we adopt the usual order on the natural numbers to the $I$-line as well as to the $K$-line. Let $\mathcal G_{p,n,l+1}$ be the set of ordered rooted trees on $l+1$ vertices such that the part containing the root lies on the $I$-line and the other part on the $K$-line. We denote an ordered tree in $\mathcal G_{p,n,{l+1}}$ by $G(i,k)$, where $i=(i_1,\dots,i_l)\in [p]^l$ and  $k=(k_1,\dots,k_l)\in [n]^l$ such that $i_1,k_1\dots,i_l,k_l,i_1$ is the canonical walk around the ordered tree $G(i,k)$ with root $i_1$.\begin{figure}[htbp] 
  \centering
     \includegraphics[width=0.9\textwidth]{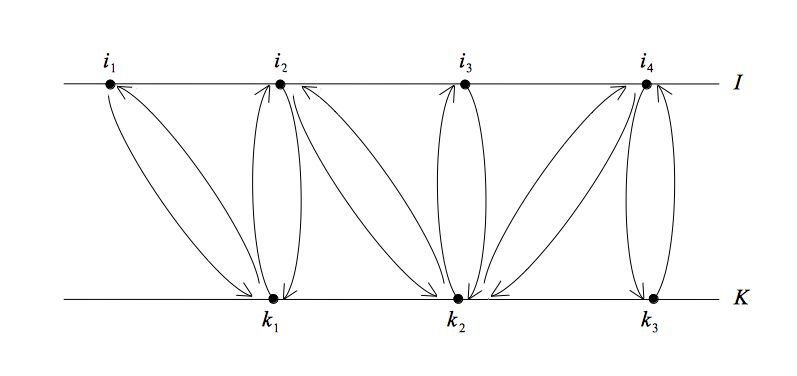}
  \caption{Depiction of an ordered tree with canonical walk $i_1,k_1,i_2,k_2,i_3,k_2,i_4,k_3,i_4,k_2,i_2,k_1,i_1$ via its two parts on the $I$- and $K$-line.} 
  \label{fig:Bild7}
  \end{figure}\\ 
An essential quantity to evaluate the $l$-th expected moment of a (classical) high-dimensional sample covariance matrix is the number of ordered trees in $\mathcal G_{p,n,{l+1}}$. The usual approach to count the number of graphs in $\mathcal G_{p,n,{l+1}}$ is to subdivide $\mathcal G_{p,n,{l+1}}$ into isomorphy classes and then to count the number of graphs in each isomorphy class. Let, $G(i,k)\in \mathcal G_{p,n,{l+1}}$ be an arbitrary ordered tree. Note that an isomorphism $\pi$ from $G(i,k)$ to an isomorphic graph $G(i',k')$ preserves the parts. Hence, we split $\pi$ into its restriction to the vertices on the $I$-line and $K$-line denoted by $\pi_I:\{i_1,\dots,i_l\}\to \{i_1',\dots,i_l'\}$ and $\pi_K:\{k_1,\dots,k_l\}\to \{k_1',\dots,k_l'\}$. Among the graphs in the isomorphy class $[G(i,k)]$ there is one graph $G(i^\text{c},k^\text{c})$ which is called the canonical representative of $[G(i,k)]$ defined as follows. The enumeration is equivalent on both parts. Therefore, we restrict to the part on the $I$-line. Let $i_1'=1$ and $1<r\le l$. If $|\{i_1,\dots,i_r\}|=|\{i_1,\dots,i_{r-1}\}|+1$ , then $i_r'=|\{i_1,\dots,i_{r-1}\}|+1$, otherwise there exists an index $s<r$ such that $i_r=i_s$ and we define $i_r'=i_s'$. Indeed, the graphs $G(i,k)$ and $G(i^\text{c},k^\text{c})$ are isomorphic and the canonical representative does not depend on the choice of the ordered tree $G(i,k)\in[G(i,k)]$. A canonical representative of a equivalence class is also called a canonical ordered tree. For $l+1\le n\vee p$ the number of equivalence classes does only depend on $l$ but not on $p$ and $n$. Now, let $G(i^c,k^c)$ be a canonical ordered tree. The number of ordered trees in $[G(i^c,k^c)]$ is given by the product of numbers of bijections from $\{k_1^c,\dots,k_l^c\}$ into subsets of the $I$-line and bijections from $\{i_1,\dots,i_l\}$ into subsets of the $K$-line. Both latter quantities depend only on the number of vertices on the $I$-line $r+1$ and are explicitly given by 
$$\frac{p!}{(p-(r+1))!} \ \ \text{and} \ \ \frac{n!}{(n-(l-r)!}.$$
Hence, two isomorphy classes have the same cardinality if their canonical representatives have the same number of vertices on the $I$-line. For fixed $r<l$ this rises the question how many canonical ordered trees have $r+1$ vertices on the $I$-line. It is well-known (see e.g. Lemma 3.4 in \cite{Bai2010}) that the answer is 
$$\frac{1}{r+1}\binom{l}{r}\binom{l-1}{r}.$$  
Alltogether, the number of ordered trees in  $\mathcal G_{p,n,{l+1}}$ is given by
$$\sum_{r=0}^{l-1}\frac{1}{r+1}\binom{l}{r}\binom{l-1}{r}\frac{p!}{(p-(r+1))!}\frac{n!}{(n-(l-r)!}.$$
Now, let us consider ordered trees with $l+1$ vertices which have a band structure on the $I$-line. This new concept will be helpful to evaluate the expected moments of banded sample covariance matrices. We say an ordered tree $G(i,k)\in \mathcal G_{p,n,{l+1}}$ is $d$-\textit{banded} (on the $I$) if the multi-index $i=(i_1,\dots,i_l)$ satisfies $|i_s-i_{s+1}|\le d$ for any $s=1,\dots,l$ with $i_{l+1}=i_1$. 
Denote the subset of all $d$-banded ordered trees in  $\mathcal G_{p,n,{l+1}}$ by $\mathcal B_{p,n,d,{l+1}}$ 
Subsequently, we assume that $d\ge l$ and $p>2ld$. The cardinality of $\mathcal B_{p,n,d,{l+1}}$ 
is crucial to evaluate the expected moments of the spectral measure of band sample covariance matrices in high dimensions. $\mathcal B_{p,n,d,{l+1}}$ 
has the same number of canonical ordered trees as $\mathcal G_{p,n,{l+1}}$, however the (asymptotic) number of isomorphic ordered trees to a canonical ordered tree does not only depend the number of vertices on the $I$-line but on the set of degrees of the vertices on the $K$-line. The later statement is investigated in the next subsection.

\subsection{Restricted compositions and the number of $d$-banded ordered trees}
A basic tool to evaluate the number of graphs in $\mathcal B_{p,n,d,{l+1}}$ isomorphic to a canonical graph $G(i^{\text{c}},k^{\text{c}})$ are compositions of natural numbers.
\begin{definition} 
For any $n\in N$, a tupel $(a_1,\dots,a_k)\in \N^k,~k\in\N$ satisfying $a_1+\dots+a_k=n$ is called a $k$-\textit{composition} of $n$. If a set $A\subset \N^k$ is designated and $(a_1,\dots,a_k)\in A$, $a_1+\dots+a_k=n$, then we name $(a_1,\dots,a_k)$ a \textit{restricted} $k$-composition. For the special case $A=\{1,\dots,m\}^k,~m\in\N$, define $F(n,k,m)$ as the number of the corresponding restricted $k$-compositions of $n$.
\end{definition}
The values $F(n,k,m),~n,k,m\in\N,$ may be determined by the method of generating functions and are explicitly given by
$$F(n,k,m)=\sum_{j=0}^k(-1)^j\binom{k}{j}\binom{n-jm-1}{k-1},$$
see \cite{Abramson1976}. Now, let $G(i^\text{c},k^\text{c})\in \mathcal B_{p,n,d,{l+1}}$ be an canonical ordered tree. The aim of this subsection is to express $|[G(i^\text{c},k^\text{c})]|$ in terms of the numbers $$F(\deg(k^c_s)d,\deg(k^c_s),2d),~k^c_s\in\{k^c_1,\dots,k^c_l\}.$$ 

\begin{lemma}\label{lemma: banded trees}
Let $G(i^\text{c},k^\text{c})\in \mathcal B_{p,n,d,{l+1}}$ with $r+1$ vertices on the $I$-line be a canonical ordered tree and $[G(i^\text{c},k^\text{c})]$ the class of isomorphic ordered trees in $\mathcal B_{p,n,d,{l+1}}$. Then,
\begin{align*}
\left\arrowvert[G(i^c,k^c)]\right\arrowvert=pn^{l-r}\prod_{k^\ast \in[l-r]}&F(\deg(k^\ast)d,\deg(k^\ast),2d)\\
&+O_l\left(n^{l-r}d^{r+1}+pn^{l-r}d^{r-1}+pd^rn^{l-r-1}\right).
\end{align*}
\end{lemma}
\begin{proof}
Let $r+1$ be the number of vertices on the $I$-line. Each pair $$(\pi_{I},\pi_{K}),\ \ \pi_{I}:[r+1]\rightarrow [p],~\pi_{K}:[l-r]\rightarrow[n]$$ of injective functions with $|\pi(i_s^\text{c})-\pi(i_{s+1}^\text{c})|\le d$ corresponds to an ordered tree $G(i,k)\in[G(i^\text{c},k^\text{c})]$ with $i_s=\pi(i_s^{\text{c}})$ and $k_s=\pi(k_s^{\text{c}}),~s=1,\dots,l$, and vice versa. As in the classical case there are $\frac{n!}{(n-(l-r))!}$ possible choices for the mapping $\pi_K$. The evaluation of the number of permissible mappings $\pi_{I}$ is more involved. First we have $p$ possible labels for the root $i_1$ of the tree. For simplification we consider only labelings with $ld < i_1< p-ld$. This reduces the number of permissible labelings by $O_l(n^{l-r}d^{r+1})$ but ensures that we do not have to take into consideration labelings close to the boundary of the $I$-line. Now, the remaining vertices on the $I$-line of the ordered tree are labeled by induction on the vertices lying on the $K$-line. For simplification of the presentation assume that $k=k^\text{c}$ ,and let $k^\ast=1,\dots,[l-r]$ run over the vertices of the ordered tree on the $K$-line. Assume that the children of the vertex $k^\ast$ on the $K$-line is already labeled for all $k^\ast<s$, $s\le l-r$. Then for $k^\ast=s$ we label its children in the following way. Let $i^{(s)}_1\le_s\dots\le_si^{(s)}_{\deg(s)-1}$ be the ordered children of $s$. The parent $i^{(s)}_\text{parent}$ of $s$ is already labeled by the inductive procedure.  By the definition of the canonical walk, a labeling of the children does not violate the $d$-band structure on the $I$-line if and only if 
\begin{align*}
\big|i^{(s)}_\text{parent}-i^{(s)}_1\big|\le d,&~\big|i^{(s)}_1-i^{(s)}_2\big|\le d,\dots,\\
&\big|i^{(s)}_{\deg(s)-2}-i^{(s)}_{\deg(s)-1}\big|\le d,~\big|i^{(s)}_{\deg(s)-1}-i^{(s)}_\text{parent}\big|\le d
\end{align*} and each label in $[p]$ is assigned to at most one already labeled vertex on the $I$-line. \begin{figure}[htbp] 
  \centering
     \includegraphics[width=0.7\textwidth]{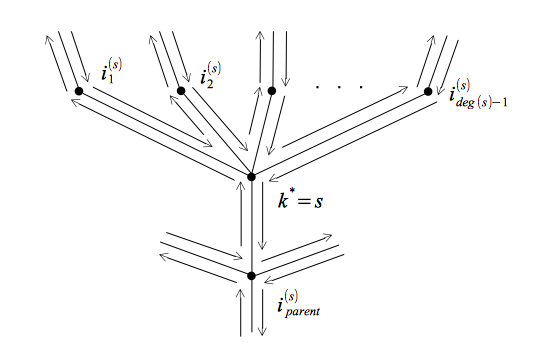}
  \caption{We label simultaneously the vertices $i_1^{(s)},i_2^{(s)},\dots,i_{\deg(s)-1}^{(s)}$. The vertices on the $I$-line ``above'' those vertices have not been labeled at this point.}
  \label{fig:Bild5}
  \end{figure} Now we simplify this problem without essentially changing the number of permissible labelings. First note that rejecting the second condition changes the number of labelings of $i^{(s)}_1\le_s\dots\le_si^{(s)}_{\deg(s)-1}$ by $O_l(d^{\deg(s)-2})$. Then for this reduced problem, it is equivalent to evaluate the number of solutions to the equation
\begin{align}\label{eq: fir sim pro}
\sum_{t=1}^{\deg(s)}b_t=0 \ \ \text{restricted to $|b_t|\le d$ for all $t=1,\dots,\deg(s)$},
\end{align}
since $b_t$ is associated with $i^{(s)}_{t-1}-i^{(s)}_{t}$, where $i^{(s)}_{0}:=i^{(s)}_{\deg(s)}:=i^{(s)}_\text{parent}$. Let $a_t=b_t+d$ for all $t=1,\dots,\deg(s)$. Then, the number of solutions to \eqref{eq: fir sim pro} is the same as to the problem
\begin{align}\label{eq: sec sim pro}
\sum_{t=1}^{\deg(s)}a_t=\deg(s)d \ \ \text{restricted to $a_t=0,\dots,2d$ for all $t=1,\dots,\deg(s)$}.
\end{align}
The numbers of solutions to \eqref{eq: sec sim pro} and to
\begin{align}\label{eq: thi sim pro}
\sum_{t=1}^{\deg(s)}a_t=\deg(s)d \ \ \text{restricted to $a_t=1,\dots,2d$ for all $t=1,\dots,\deg(s)$}
\end{align}
differ by $O_l(d^{\deg(s)-2})$. Putting all those labelings together proves the claim. 
\end{proof}



\section{Main result}\label{section: main}
Now we are able to state and prove the main result of this article. 
\begin{theorem}\label{the: band}
Let $\big(\boldsymbol X^{(p,n)}\big)_{p,n}$ be a sequence of random matrices $\boldsymbol X^{(p,n)}\in\R^{p\times n}$ with independent entries $X^{(p,n)}_{ik}$  with mean $0$ and variance $1$. Additionally, suppose that for any $\eta>0$,
\begin{align}
\frac{1}{\eta^2np}\sum_{i=1}^p\sum_{k=1}^n\E\left(\left|X^{(p,n)}_{i,k}\right|^2\ind\left\{\left|X^{(p,n)}_{i,k}\right|\ge\eta\sqrt{n}\right\}\right)\longrightarrow 0. \label{eq: mom}
\end{align}
Denote 
\begin{align}
\boldsymbol S^{(p,n)}:= \left(\frac{1}{n}\boldsymbol X^{(p,n)}\left(\boldsymbol X^{(p,n)}\right)'\right)\circ \boldsymbol1_{d}. \label{eq: S}
\end{align}
Then, the sequence of empirical spectral distributions $\mu^{(p,n)}=p^{-1}\sum_{i=1}^p\delta_{\lambda_i(S^{(p,n)})}$ almost surely converges weakly to a measure $\mu$, as $n\to \infty$ or $p\to\infty$, while $d\to\infty$, $\frac{n}{p}\to 0,$ and $\frac{2d}{n}\to y>0$. The $l$-th moment $m_l,l\in\N,$ of the limiting spectral distribution is given by
\begin{equation*}
m_l=\sum_{G(i^c,k^c)}\prod_{k^\ast}\sum_{j=0}^{\deg(k^\ast)}\ind\{k^\ast>2j\}(-1)^j(y(\deg(k^\ast-2j)))^{\deg(k^\ast)-1}\frac{\deg(k^\ast)}{j!(\deg(k^\ast)-j)!},
\end{equation*}
where the outer sum runs over all canonical ordered trees $G(i^c,k^c)$ with $l+1$ vertices such that the part containing the root lies on the $I$-line and the other part on the $K$-line, and the product runs over all vertices $k^\ast\in\{k_1^c,\dots,k_l^c\}$.
\end{theorem}
\noindent Note that by bounding the quantities $F(n,k,m)$ by $m^{k-1}$ it follows from the proof of Theorem \ref{the: band}
$$m_l\le \sum_{r=0}^{l-1}\frac{1}{r+1}\binom{l}{r}\binom{l-1}{r}y^{r}.$$
The right hand-side is the $l$-th moment of the Mar\v{c}enko-Pastur distribution with parameter $y$ which implies that the random variable $X\sim \mu$ is bounded in absolute value by $(1+\sqrt{y})^2$ since
$$\P\left(|X|\ge x\right)\le \frac{\E X^{2l}}{x^{2l}}\le \frac{1}{x^{2l}}\sum_{r=0}^{2l-1}\frac{1}{r+1}\binom{2l}{r}\binom{2l-1}{r}y^{r}\overset{l\to\infty}{\longrightarrow} 0$$
for any $x>(1+\sqrt{y})^2$. Beyond that, so far there is nothing known about the distribution $\mu$. Especially, the natural questions, for which $y>0$ the random variable $X$ is negative with positive probability, and whether the bound $(1+\sqrt{y})^2$ on the support is sharp, are open. The answers to both problems are essential groundwork to understand the asymptotical behavior of the extreme eigenvalues of high-dimensional banded sample covariance matrices, cf. \cite{Bai1993} for the almost sure limits of the extreme eigenvalues of high-dimensional sample covariance matrices. 
\begin{figure}[htbp] 
  \centering
     \includegraphics[width=0.8\textwidth]{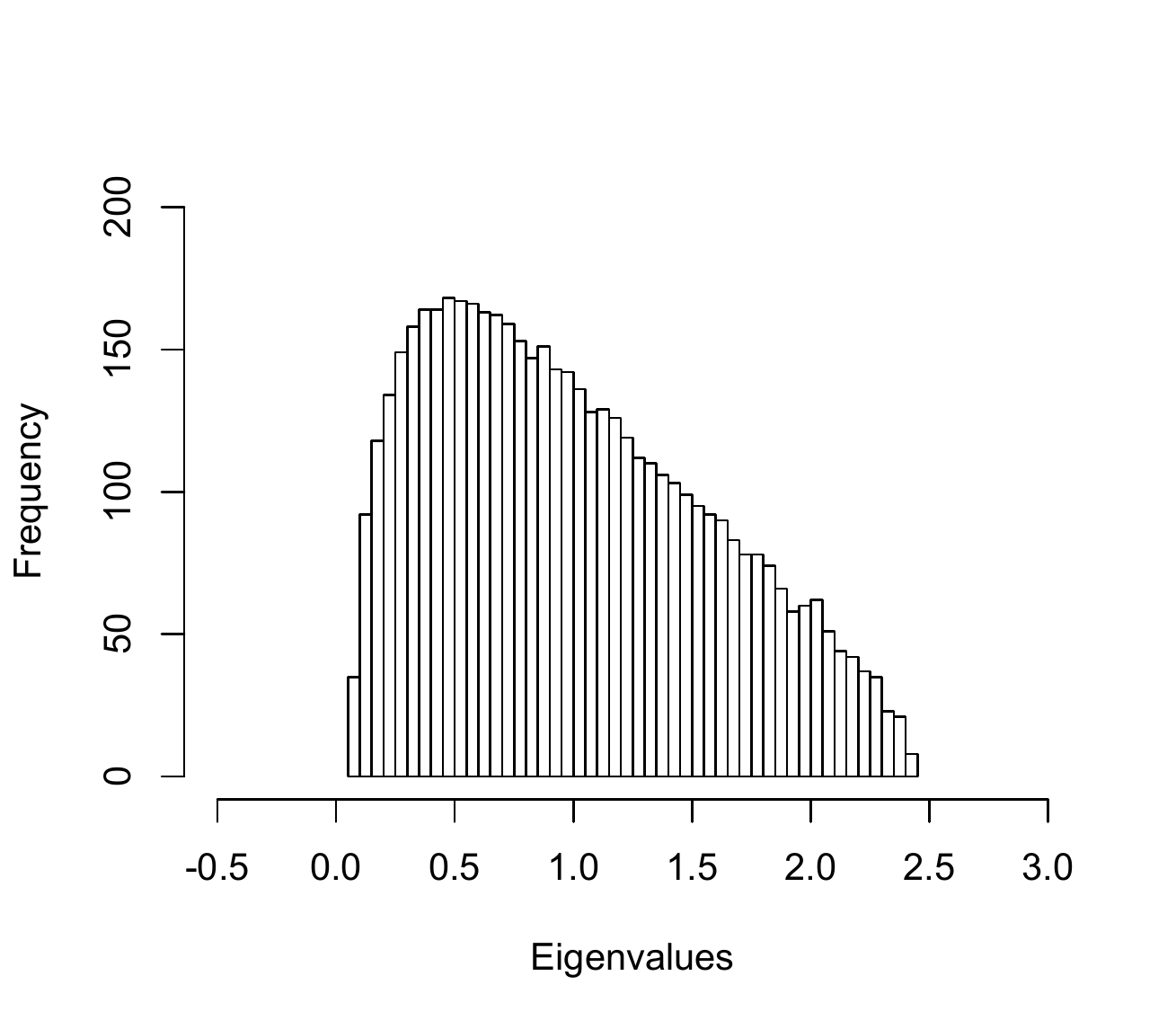}
  \caption{Histogram of the eigenvalues of a $5000\times 5000$ sample covariance with bandwidth $100$ based on $600$ samples from the standard normal distribution} 
  \label{fig:plot 1}
  \end{figure}
    \begin{figure}[htbp] 
  \centering
     \includegraphics[width=\textwidth]{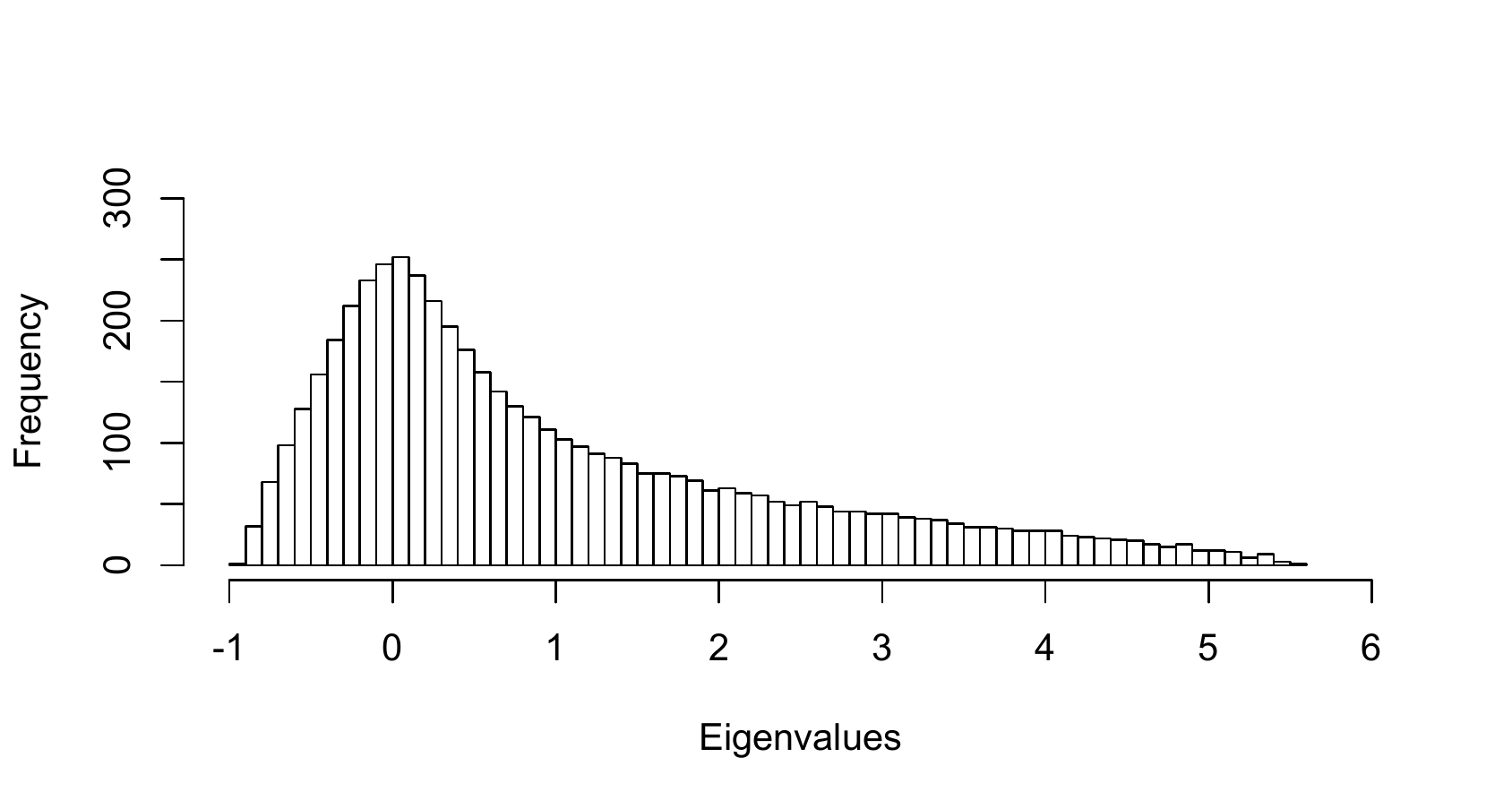}
  \caption{Histogram of the eigenvalues of a $5000\times 5000$ sample covariance with bandwidth $300$ based on $300$ samples from the standard normal distribution} 
  \label{fig:plot 2}
  \end{figure}
  \begin{proof}
For ease of notation write $\boldsymbol X$ and $\boldsymbol S$ instead of  $\boldsymbol X^{(p,n)}$ and $\boldsymbol S^{(p,n)}$. Accordingly, the entries of $\boldsymbol X$ and $\boldsymbol S$ are denoted by $X_{ik}$ and $S_{ik}$. Since in the situation of the theorem almost sure convergence for $p\to\infty$ is a stronger statement than for $n\to\infty$, we restrict the proof to the case $p\to\infty$, while $\frac{n}{p}\to 0,$ and $\frac{2d}{n}\to y$. First, we choose a sequence $\eta_n\to 0$ such that
$$
\frac{1}{\eta_n^2np}\sum_{i=1}^p\sum_{k=1}^n\E\left(\left|X^{(p,n)}_{i,k}\right|^2\ind\left\{\left|X^{(p,n)}_{i,k}\right|\ge\eta_n\sqrt{n}\right\}\right)\longrightarrow 0,
$$
and $\eta_n\ge \frac{1}{\log n}$, $n\ge 2$. As in the proof of the Theorem 3.10 in \cite{Bai2010} we start with the step of truncation, centralization and standardization which allows to work with a simplified matrix afterwards. 
Whereas the arguments for truncation at the level $\eta_n\sqrt{n}$ may be transferred almost analogously to the matrix $\boldsymbol S$, the arguments for centralization and standardization need to be refined. Then, the convergence of the expected moments of the empirical spectral distribution is shown by means of Lemma \ref{lemma: banded trees} which is crucial at this step. Finally, we have to prove that the fluctuation of the moments of the empirical spectral distribution almost surely converge to zero. This may be done similarly as in \cite{Bai2010}  for Wigner matrices.
\subsection{Truncation, centralization and standardization}
Let $\boldsymbol{\tilde X}$ be the matrix with entries $\tilde X_{ik}:= X_{ik}\ind \{|X_{ik}|\le \eta_n \sqrt{n}\}$ and $\boldsymbol{\tilde S}$ be the matrix defined by the right hand side of \eqref{eq: S}, where $\boldsymbol X$ is replaced by $\boldsymbol {\tilde X}$. Then,
\begin{align}
d_K\left( \mu^{\boldsymbol S},\mu^{\boldsymbol {\tilde S}} \right)&\le \frac{1}{p}\rank\left( \boldsymbol S-\boldsymbol{\tilde S}\right)\\
&\le \frac{1}{p}\rank\left( \left(\left(\boldsymbol X-\boldsymbol{\tilde X}\right)\boldsymbol X'+\boldsymbol {\tilde X} \left(\boldsymbol X-\boldsymbol{\tilde X}\right)'\right)\circ \boldsymbol{1}_d\right)\\
&\le \frac{2}{p}\sum_{i=1}^p\sum_{k=1}^n \ind\left\{\left|X_{ik}\right|>\eta_n\sqrt{n}\right\},
\end{align}
where Theorem \ref{theorem: A43} is used in the first line, and the last line follows by subadditivity of the rank and the fact that $\boldsymbol X-\boldsymbol {\tilde X}$ has no more than 
$$\sum_{i=1}^p\sum_{k=1}^n \ind\left\{\left|X_{ik}\right|>\eta_n\sqrt{n}\right\}$$
non-zero rows. Hence, it remains to prove that
\begin{align}
\frac{2}{p}\sum_{i=1}^p\sum_{k=1}^n \ind\left\{\left|X_{ik}\right|>\eta_n\sqrt{n}\right\}\overset{a.s.}{\longrightarrow} 0 \ \ \text{as }p\to\infty.
\end{align}
Analogously to page 27 in \cite{Bai2010} we have by \eqref{eq: mom}
\begin{align}
\frac{2}{p}\sum_{i=1}^p\sum_{k=1}^n \E\ind\left\{\left|X_{ik}\right|>\eta_n\sqrt{n}\right\}=o(1)
\end{align}
and 
\begin{align}
\Var\left(\frac{2}{p}\sum_{i=1}^p\sum_{k=1}^n\ind\left\{\left|X_{ik}\right|>\eta_n\sqrt{n}\right\}\right)=o\left(\frac{1}{p}\right),
\end{align}
such that by Bernstein's inequality and the Borel-Cantelli lemma for any $\varepsilon>0$
\begin{align}
\limsup_{p \to\infty} \frac{2}{p}\sum_{i=1}^p\sum_{k=1}^n\ind\left\{\left|X_{ik}\right|>\eta_n\sqrt{n}\right\}\le \varepsilon \ \ \text{a.s.}
\end{align}
Therefore,
\begin{align}
d_K\left(F^{\boldsymbol S},F^{\boldsymbol {\tilde S}} \right)\overset{\text{a.s.}}{\longrightarrow} 0 \ \ \text{as }p\to\infty.
\end{align}
Redefine $\boldsymbol{\tilde X}$ by $\boldsymbol X$ and $\boldsymbol {\tilde S}$ by $\boldsymbol S$. Now, we prove that we may recenter the entries of the matrix $\boldsymbol X$. We have
\begin{align}
&d_L^2\left(F^{\boldsymbol S},F^{\left(\frac{1}{n}\left(\boldsymbol X-\E \boldsymbol X\right)\left(\boldsymbol X -\E \boldsymbol X\right)'\right)\circ \boldsymbol 1_d}\right)\\
&\hspace{0.5cm}\le 2d_L^2\left(F^{\boldsymbol S},F^{\left(\frac{1}{n}\boldsymbol X\boldsymbol X'+\frac{1}{n}\boldsymbol X\E\boldsymbol X'+\frac{1}{n}\left(\E\boldsymbol X\right)\boldsymbol X'\right)\circ \boldsymbol 1_d}\right)\label{eq: cen fir}\\
&\hspace{1cm}+2d_L^2\left(F^{\left(\frac{1}{n}\boldsymbol X\boldsymbol X'+\frac{1}{n}\boldsymbol X\E\boldsymbol X'+\frac{1}{n}\left(\E\boldsymbol X\right)\boldsymbol X'\right)\circ \boldsymbol 1_d},F^{\left(\frac{1}{n}\left(\boldsymbol X-\E \boldsymbol X\right)\left(\boldsymbol X -\E \boldsymbol X\right)'\right)\circ \boldsymbol 1_d}\right).\label{eq: cen sec}
\end{align}
By Corollary \ref{corollary: A41} for the term \eqref{eq: cen sec} holds
\begin{align}
\eqref{eq: cen sec}^{3/2}&\le \frac{2\sqrt{2}}{p}\tr\left[\left(\frac{1}{n}\left(\E \boldsymbol X\left(\E \boldsymbol X' \right)\right)\circ \boldsymbol 1_d\right)^2\right]\\
&=\frac{2\sqrt{2}}{pn^2}\sum_{i,j=1}^p\left(\sum_{k=1}^n \E X_{ik}\E X_{jk}\right)^2\ind\{|i-j|\le d\}\\
&\le \frac{2\sqrt{2}(2d+1)}{n^2\eta_n^{4}}\rightarrow 0,
\end{align}
where the last line follows by the inequality $|\E X_{ik}|\le \eta_n^{-1}n^{-1/2}$.
To evaluate the term \eqref{eq: cen fir} we combine Corollary \ref{corollary: A41} and Theorem \ref{theorem: A43}. Thereto, we prove first that there are not to many rows $i$ in the matrix $\boldsymbol X$ which suffice 
\begin{align}
\sum_{j=1}^p\left(\sum_{k=1}^n X_{ik}\E X_{jk}\right)^2\ind\left\{|i-j|\le d\right\}\ge \frac{n}{\eta_n^5}.
\end{align}
By the union bound and Markov's inequality,
\begin{align}
&\P\left(\sum_{j=1}^p\left(\sum_{k=1}^n X_{ik}\E X_{jk}\right)^2\ind\left\{|i-j|\le d\right\}\ge \frac{n}{\eta_n^5}\right)\\
&\hspace{0.5cm}\le \P\left(\sum_{j=1}^p\sum_{k=1}^n X_{ik}^2\left(\E X_{jk}\right)^2\ind\left\{|i-j|\le d\right\}\ge \frac{n}{2\eta_n^5}\right)\\
&\hspace{0.8cm}+\P\left(\sum_{j=1}^p\sum_{\substack{k_1,k_2=1\\ k_1\neq k_2}}^n X_{ik_1}X_{ik_2}\E X_{jk_1}\E X_{jk_2}\ind\left\{|i-j|\le d\right\}\ge \frac{n}{2\eta_n^5}\right)\\
&\hspace{0.5cm}\le \frac{2\eta_n^5\E\left[\sum_{j=1}^p\sum_{k=1}^n X_{ik}^2\left(\E X_{jk}\right)^2\ind\left\{|i-j|\le d\right\}\right]}{n}\\
&\hspace{0.8cm}+\frac{4\eta_n^{10}\E\left[\left(\sum_{j=1}^p\sum_{k_1\neq k_2} X_{ik_1}X_{ik_2}\E X_{jk_1}\E X_{jk_2}\ind\left\{|i-j|\le d\right\}\right)^2\right]}{n^2}\\
&\hspace{0.5cm}\le \frac{2(2d+1)\eta_n^3}{n}+\frac{8\eta_n^6(2d+1)^2}{n^2}\left(2+\frac{4}{\eta_n^2}+\frac{1}{\eta_n^4}\right).
\end{align}
Thus, by Hoeffding's inequality for sufficiently large $p$
\begin{align*}
&\P\left(\sum_{i=1}^p\ind\left\{\sum_{j=1}^p\left(\sum_{k=1}^n X_{ik}\E X_{jk}\right)^2\ind\left\{|i-j|\le d\right\}\ge \frac{n}{\eta_n^5}\right\}\ge p\sqrt{\eta_n}\right)\\
&\hspace{0.5cm}\le \P\Bigg[\sum_{i=1}^p\Bigg(\ind\Bigg\{\sum_{j=1}^p\left(\sum_{k=1}^n X_{ik}\E X_{jk}\right)^2\ind\left\{|i-j|\le d\right\}\ge\frac{n}{\eta_n^5}\Bigg\}\\
&\hspace{0.8cm}-\E\ind\Bigg\{\sum_{j=1}^p\left(\sum_{k=1}^n X_{ik}\E X_{jk}\right)^2\ind\left\{|i-j|\le d\right\}\ge\frac{n}{\eta_n^5}\Bigg\}\Bigg)\\
&\hspace{1.1cm}\ge p\left(\sqrt{\eta_n}-\frac{4d\eta_n}{n}-\frac{16\eta_n^2(2d+1)^2}{n^2}\right)\Bigg]\\
&\hspace{0.5cm}\le \exp\left(-2p\left(\sqrt{\eta_n}-\frac{4d\eta_n}{n}-\frac{16\eta_n^2(2d+1)^2}{n^2}\right)^2\right).
\end{align*}
The last line is summable over $p$. Hence, by the Borel-Cantelli lemma 
\begin{align}
\frac{1}{p}\sum_{i=1}^p\ind\left\{\sum_{j=1}^p\left(\sum_{k=1}^n X_{ik}\E X_{jk}\right)^2\ind\left\{|i-j|\le d\right\}\ge \frac{n}{\eta_n^5}\right\}\overset{\text{a.s.}}{\longrightarrow} 0 \ \ \text{as }p\to\infty.\label{eq: con}
\end{align}
Let $\boldsymbol{\tilde X}\in\R^{p\times n}$ be the matrix with entries 
\begin{align}
\tilde X_{ik}:=X_{ik}\ind\left\{\sum_{j=1}^p\left(\sum_{k=1}^n X_{ik}\E X_{jk}\right)^2\ind\left\{|i-j|\le d\right\}< \frac{n}{\eta_n^5}\right\}.
\end{align}
Then, we obtain
\begin{align*}
\eqref{eq: cen fir}&\le 4d_L^2\left(F^{\boldsymbol S},F^{\left(\frac{1}{n}\boldsymbol X\boldsymbol X'+\frac{1}{n}\boldsymbol{\tilde X}\E\boldsymbol X'+\frac{1}{n}\left(\E\boldsymbol X\right)\boldsymbol{\tilde X}'\right)\circ \boldsymbol 1_d}\right)\\
&\hspace{0.5cm}+4d_L^2\left(F^{\left(\frac{1}{n}\boldsymbol X\boldsymbol X'+\frac{1}{n}\boldsymbol{X}\E\boldsymbol X'+\frac{1}{n}\left(\E\boldsymbol X\right)\boldsymbol{X}'\right)\circ \boldsymbol 1_d},F^{\left(\frac{1}{n}\boldsymbol X\boldsymbol X'+\frac{1}{n}\boldsymbol{\tilde X}\E\boldsymbol X'+\frac{1}{n}\left(\E\boldsymbol X\right)\boldsymbol{\tilde X}'\right)\circ \boldsymbol 1_d}\right)\\
&\le 4\left(\frac{1}{p}\tr\left[\left(\frac{1}{n}\left(\boldsymbol{\tilde X}\E \boldsymbol X'+(\E \boldsymbol X)\boldsymbol{\tilde X}'\right)\circ\boldsymbol 1_d\right)^2\right]\right)^{2/3}\\
&\hspace{0.5cm}+4\left(\frac{1}{p}\rank\left[\left(\left(\boldsymbol{\tilde X}-\boldsymbol X\right)\E \boldsymbol X'+\E \boldsymbol X\left(\boldsymbol{\tilde X}-\boldsymbol X\right)'\right)\circ\boldsymbol 1_d\right]\right)^2
\end{align*}
By \eqref{eq: con} the summand in the last line vanishes asymptotically almost surely, whereas for the first term we have
\begin{align}
 &\frac{1}{p}\tr\left[\left(\frac{1}{n}\left(\boldsymbol{\tilde X}\E \boldsymbol X'+(\E \boldsymbol X)\boldsymbol{\tilde X}'\right)\circ\boldsymbol 1_d\right)^2\right]\\
 &\hspace{0.5cm}\le \frac{4}{pn^2}\sum_{i,j=1}^p\left(\sum_{k=1}^n\tilde X_{ik}\E X_{jk}\right)^2\ind\left\{|i-j|\le d\right\}\le \frac{4}{n\eta_n^5}\rightarrow 0.
\end{align}
So,
\begin{align}
d_L\left(F^{\boldsymbol S},F^{\left(\frac{1}{n}\left(\boldsymbol X-\E \boldsymbol X\right)\left(\boldsymbol X -\E \boldsymbol X\right)'\right)\circ \boldsymbol 1_d}\right)\overset{\text{a.s.}}{\longrightarrow}0\ \ \text{as }p\to\infty.
\end{align}
Subsequently denote $\boldsymbol X-\E \boldsymbol X$ by $\boldsymbol X$. It remains to standardize the matrix $\boldsymbol X$. In fact, we do not standardize all entries of $\boldsymbol X$ but those with  
\begin{align}
\sigma_{ik}^2:= \E X_{ik}^2> \eta_n. \label{eq: low bou}
\end{align} 
In particular, by condition \eqref{eq: mom} only $o(np)$ entries do not satisfy \eqref{eq: low bou}. Without loss of generality we may assume that either
\begin{align*}
n\le \frac{p}{(\log p)^2} \ \ \text{or} \ \ n\ge \frac{p}{(\log p)^2}
\end{align*}
holds on the whole sequence. Let $\boldsymbol{\tilde S}$ be the matrix with entries $\tilde S_{ii}:=S_{ii}$, $i=1,...,p$, and 
$$ \tilde S_{ij}:=\frac{1}{n}\sum_{k=1}^n \tilde\sigma_{ik}^{-1}\tilde\sigma_{jk}^{-1} X_{ik}X_{jk}\ind\{|i-j|\le d\},~i\neq j.$$
where
$$\tilde\sigma_{ik}:=\tilde\sigma_{ik,n}:=\sigma_{ik}\ind\{\sigma_{ik}^2> \eta_n\}+\ind\{\sigma_{ik}^2\le \eta_n\}.$$
First consider the case $n\ge p/(\log p)^2$. Define
\begin{align*}
I_1&:=\frac{1}{pn^2}\sum_{\substack{i,j=1\\ i\neq j}}^p\sum_{k=1}^n \left(\tilde\sigma_{ik}^{-1}\tilde\sigma_{jk}^{-1}-1\right)^2 \left (X_{ik}^2-\E X_{ik}^2\right)\left(X_{jk}^2-\E X_{jk}^2\right)\ind\{|i-j|\le d\}\\
I_2&:=\frac{2}{pn^2}\sum_{\substack{i,j=1\\ i\neq j}}^p\sum_{k=1}^n \left(\tilde\sigma_{ik}^{-1}\tilde\sigma_{jk}^{-1}-1\right)^2  \left (X_{ik}^2-\E X_{ik}^2\right) \E X_{jk}^2\ind\{|i-j|\le d\}\\
I_3&:=\frac{1}{pn^2}\sum_{\substack{i,j=1\\ i\neq j}}^p\sum_{k=1}^n\left(\tilde\sigma_{ik}^{-1}\tilde\sigma_{jk}^{-1}-1\right)^2  \E X_{ik}^2\E X_{jk}^2\ind\{|i-j|\le d\}\\
I_4&:=\frac{1}{pn^2}\sum_{\substack{i,j=1\\ i\neq j}}^p\sum_{\substack{k_1,k_2=1\\ k_1\neq k_2}}^n \left(\tilde\sigma_{ik_1}^{-1}\tilde\sigma_{jk_1}^{-1}-1\right)\left(\tilde\sigma_{ik_2}^{-1}\tilde\sigma_{jk_2}^{-1}-1\right) 
X_{ik_1}X_{jk_1}X_{ik_2}X_{jk_2}\\
&\hspace{5cm}\times\ind\{|i-j|\le d\}.
\end{align*}
By Corollary \ref{corollary: A41} we have
\begin{align*}
d_L^3\left(F^{\boldsymbol S},F^{\boldsymbol{\tilde S}}\right)&\le \frac{1}{pn^2}\sum_{\substack{i,j=1\\ i\neq j}}^p\left(\sum_{k=1}^n \left(\tilde\sigma_{ik}^{-1}\tilde\sigma_{jk}^{-1}-1\right) X_{ik}X_{jk}\ind\{|i-j|\le d\} \right)^2\\
&=I_1+I_2+I_3+I_4.
\end{align*}
First note that by \eqref{eq: mom} and by the inequality $\E X_{ik}^2\le 1$ the term $I_3$ satisfies
\begin{align*}
I_3&\le \frac{1}{pn^2}\sum_{\substack{i,j=1\\i\neq j}}^p\sum_{k=1}^n\left(1-\sigma_{ik}\sigma_{jk}\right)^2\ind\{|i-j|\ge d\}\\
&\le \frac{1}{pn^2}\sum_{\substack{i,j=1\\i\neq j}}^p\sum_{k=1}^n\left(1-\E X_{ik}^2\E X_{jk}^2\right)\ind\{|i-j|\ge d\}\\
&= \frac{1}{pn^2}\sum_{\substack{i,j=1\\i\neq j}}^p\sum_{k=1}^n\left(1-\E X_{ik}^2+\E X_{ik}^2\left(1-\E X_{jk}^2\right)\right)\ind\{|i-j|\ge d\}\\
&\le \frac{4d}{pn^2}\sum_{i=1}^p\sum_{k=1}^n(1-\E X_{ik}^2)\rightarrow 0. 
\end{align*}
Let $\varepsilon>0$ and $p\in\N$ sufficiently large such that $I_3\le \varepsilon/4$. Then,
\begin{align*}
&\P\left(d_L\left(F^{\tilde S},F^S\right)\ge \varepsilon\right)\le \P\left(I_1\ge\frac{\varepsilon}{4}\right)+\P\left(I_2\ge\frac{\varepsilon}{4}\right)+\P\left(I_4\ge\frac{\varepsilon}{4}\right).
\end{align*}
We will use Markov's inequality to bound each of the three probabilities on the right hand side. Denote $Y_{ik}:=X_{ik}^2-\E X_{ik}^2$ and observe that
\begin{align}\label{eq: exp Y}
\E |Y_{ik}|^m\le 2\eta_n^{2(m-1)}n^{m-1},~~m\in\N.
\end{align}
Then for $p$ sufficiently large,
\begin{align*}
\E|I_1|^4&\le \frac{1}{\eta_n^{8}n^8p^4}\sum_{\substack{i_1,j_1=1\\ i_1\neq j_1}}^p\sum_{\substack{i_2,j_2=1\\ i_2\neq j_2}}^p\sum_{\substack{i_3,j_3=1\\ i_3\neq j_3}}^p\sum_{\substack{i_4,j_4=1\\ i_4\neq j_4}}^p\sum_{\substack{k_1,k_2,\\k_3,k_4=1}}^n\left|\E  \prod_{l=1}^4Y_{i_lk_l}Y_{j_lk_l}\ind\{|i_l-j_l|\le d\}\right| \\
&\le\frac{1}{\eta_n^{8}n^8p^4}\sum_{\substack{i_1,j_1=1\\ i_1\neq j_1}}^p\sum_{\substack{i_2,j_2=1\\ i_2\neq j_2}}^p\sum_{\substack{i_3,j_3=1\\ i_3\neq j_3}}^p\sum_{\substack{i_4,j_4=1\\ i_4\neq j_4}}^p\sum_{k=1}^n\left|\E  \prod_{l=1}^4Y_{i_lk}Y_{j_lk}\ind\{|i_l-j_l|\le d\}\right| \\
&\hspace{0.5cm}+\frac{3}{\eta_n^{8}n^8p^4}\sum_{\substack{i_1,j_1=1\\ i_1\neq j_1}}^p\sum_{\substack{i_2,j_2=1\\ i_2\neq j_2}}^p\sum_{\substack{i_3,j_3=1\\ i_3\neq j_3}}^p\sum_{\substack{i_4,j_4=1\\ i_4\neq j_4}}^p\sum_{\substack{k_1,k_2=1 \\ k_1\neq k_2}}^n\left\arrowvert \E  Y_{i_1k_1}Y_{j_1k_1} Y_{i_2k_1}Y_{j_2k_1} \right\arrowvert \\  
&\hspace{4cm}\times\left\arrowvert  \E Y_{i_3k_2}Y_{j_3k_2}  Y_{i_4k_2}Y_{j_4k_2} \right\arrowvert \prod_{l=1}^4\ind\{|i_l-j_l|\le d\}\\
&\le\frac{8}{\eta_n^{8}n^8p^4}\sum_{\substack{i,j=1\\ i\neq j}}^p\sum_{k=1}^n\E  Y_{ik}^4 \E Y_{jk}^4\\
&\hspace{0.5cm}+\frac{1}{\eta_n^{8}n^8p^4}\underset{3\le |\{i_1,i_2,i_3,i_4,j_1,j_2,j_3,j_4\}|\le 4}{\sum_{\substack{i_1,j_1=1\\ i_1\neq j_1}}^p\sum_{\substack{i_2,j_2=1\\ i_2\neq j_2}}^p\sum_{\substack{i_3,j_3=1\\ i_3\neq j_3}}^p\sum_{\substack{i_4,j_4=1\\ i_4\neq j_4}}^p}\sum_{k=1}^n\left|\E  \prod_{l=1}^4Y_{i_lk}Y_{j_lk}\right|\\
&\hspace{0.5cm}+\frac{12}{\eta_n^{8}n^8p^4}\sum_{\substack{i_1,j_1=1\\ i_1\neq j_1}}^p\sum_{\substack{i_2,j_2=1\\ i_2\neq j_2}}^p\sum_{\substack{k_1,k_2=1 \\ k_1\neq k_2}}^n \E  Y_{i_1k_1}^2\E Y_{j_1k_1}^2  \E Y_{i_2k_2}^2\E Y_{j_2k_2}^2 \\
&\le C\left(\frac{\eta_n^4}{np^2} + \frac{\eta_n^2}{n^2}  +\frac{1}{n^2}\right),
\end{align*}
where $C>0$ is an absolute constant. 
Further,
\begin{align*}
\E|I_2|^4&\le \frac{8}{\eta_n^{8}n^8}\sum_{i_1,i_2,i_3,i_4=1}^p\sum_{k_1,k_2,k_3,k_4=1}^n\left|\E Y_{i_1k_1} Y_{i_2k_2}Y_{i_3k_3} Y_{i_4k_4}\right| \\
&\le \frac{16}{\eta_n^{8}n^8}\sum_{\substack{i=1}}^p\sum_{k=1}^n\E Y_{ik}^4+\frac{48}{\eta_n^{8}n^8}\sum_{\substack{i_1,i_2=1 \\ i_1\neq i_2}}^p\sum_{\substack{k=1}}^n\E Y_{i_1k_1}^2 \E Y_{i_2k_2}^2\\
&\hspace{1cm}+\frac{48}{\eta_n^{8}n^8}\sum_{\substack{i_1,i_2=1}}^p\sum_{\substack{k_1,k_2=1 \\ k_1\neq k_2}}^n\E Y_{i_1k_1}^2 \E Y_{i_2k_2}^2\\
&\le\frac{32p}{\eta_n^{2}n^4}+\frac{96p^2}{\eta_n^{4}n^5}+\frac{96p^2}{\eta_n^{4}n^4}
\end{align*}
and,
\begin{align*}
\E|I_4|^2&\le \frac{1}{\eta_n^8p^2n^4}\sum_{\substack{i_1,j_1=1\\i_1\neq j_1}}^p\sum_{\substack{i_2,j_2=1\\i_2\neq j_2}}^p\sum_{\substack{k_1,k_2=1 \\ k_1\neq k_2}}^n\sum_{\substack{k_3,k_4=1 \\ k_3\neq k_4}}^n\Big|\E X_{i_1k_1}X_{j_1k_1}X_{i_1k_2}X_{j_1k_2}\\
&\hspace{3cm}\times X_{i_2k_3}X_{j_2k_3}X_{i_2k_4}X_{j_2k_4} \Big| \\
&\le  \frac{4}{\eta_n^8p^2n^4}\sum_{\substack{i_1,j_1=1\\i_1\neq j_1}}^p\sum_{\substack{k_1,k_2=1 \\ k_1\neq k_2}}^n\E X_{i_1k_1}^2\E X_{i_1k_2}^2 \E X_{j_1k_1}^2\E X_{j_1k_2}^2 \\
&\le \frac{4}{\eta_n^8n^2}.
\end{align*}
Each of the three expressions is summable over $p$. Therefore, it holds that almost surely for $p$ sufficiently large $d_L(F^{\boldsymbol{\tilde S}},F^{\boldsymbol S})<\varepsilon$. This implies  $d_L(F^{\boldsymbol{\tilde S}},F^{\boldsymbol S})\rightarrow 0$ almost surely as $p\to\infty.$\\
 The arguments used here for $n\ge p/(\log p)^2$ are also applicable to the general case. However, this requires to evaluate the fourth moments of $I_1$ and $I_2$ more carefully and to deduce an appropriate bound for $\E |I_4|^4$. Possibly, the following arguments are more suitable for the case $n\le p/(\log p)^2$.\\ 
The essential idea is to cover the band structure of the matrix by a composed block structure, and then to exploit the independency of the submatrices of $\boldsymbol{S}-\boldsymbol{\tilde S}$ corresponding to a single block structure. Thereto, define index sets 
$$V^{(1)}_k=\{2kd-2d+k,2kd-2d+k+1,\dots,2(k+1)d-2d+k\}, \ \ k=1,...,\lfloor p/(2d+1)\rfloor,$$
and
$$V^{(2)}_k=\{2kd-d+k,2kd-d+k+1,\dots,2(k+1)d-d+k\}, \ \ k=1,...,\lfloor (p-d)/(2d+1)\rfloor,$$
Note that at most $2d$ rows in the lower right corner of the matrix might be not covered by the composed block structure.
  \begin{figure}[htbp] 
  \centering
     \includegraphics[width=0.5\textwidth]{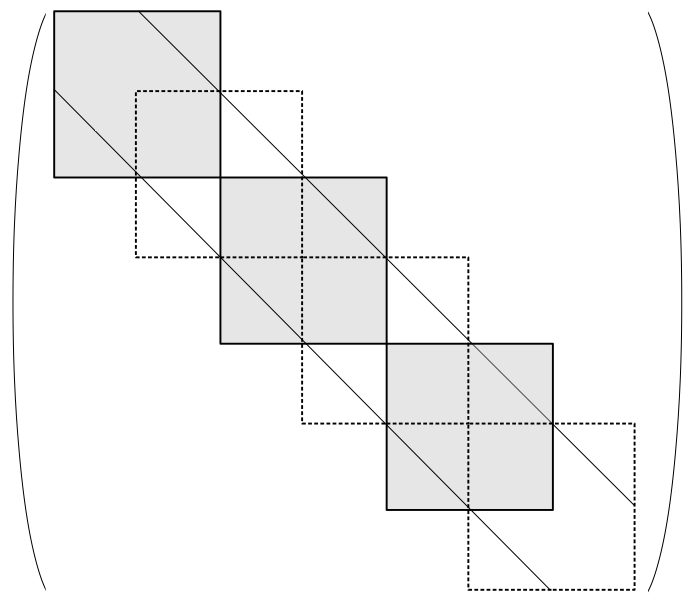}
  \caption{Band structure covered by two block structures} 
  \label{fig:Matrix}
  \end{figure}

\noindent Let $\boldsymbol{\tilde S}^{(1)}_k,{\boldsymbol S}^{(1)}_k\in\R^{(2d+1)\times (2d+1)}$,~$k=1,...,\lfloor p/(2d+1)\rfloor$ be the submatrices of $\boldsymbol {\tilde S}$ and $\boldsymbol S$ corresponding to the indices $V^{(1)}_k\times V^{(1)}_k$. Analogously, define the matrices $\boldsymbol{\tilde S}^{(2)}_k,\boldsymbol{S}^{(2)}_k\in\R^{(2d+1)\times (2d+1)}$ for $k=1,...,\lfloor (p-d)/(2d+1)\rfloor$. Then it holds for any $l\le 1,...,\lfloor p/(2d+1)\rfloor $, 
\begin{align}
\E\tr\left(\left(\boldsymbol{\tilde S}_l^{(1)}-\boldsymbol{S}_l^{(1)}\right)^2\right)&\le \frac{1}{n^2}\sum_{\substack{i,j\in V^{(1)}_l \\ i\neq j}}\sum_{k=1}^n\left(\tilde\sigma_{ik}^{-1}\tilde\sigma_{jk}^{-1}-1\right)^2\E X_{ik}^2\E X_{jk}^2 \\
&\le \frac{2d+1}{n}\frac{2}{n}\sum_{i\in V_l^{(1)}}\sum_{k=1}^n 1-\E X_{ik}^2, \label{eq: block A}
\end{align}
and analogously for $l=1,...,\lfloor (p-d)/(2d+1)\rfloor$
\begin{align*}
\E\tr\left(\left(\boldsymbol{\tilde S}_l^{(2)}-\boldsymbol{S}_l^{(2)}\right)^2\right)\le  \frac{2d+1}{n}\frac{2}{n}\sum_{i\in V_l^{(2)}}\sum_{k=1}^n 1-\E X_{ik}^2.
\end{align*}
We conclude by condition \eqref{eq: mom}, inequality \eqref{eq: block A}, and Markov's and Hoeffding's inequality for any $\varepsilon>0$ and $p$ sufficiently large,
\begin{align*}
&\P\left(\sum_{l=1}^{\lfloor p/(2d+1)\rfloor }\ind\left\{\tr\left(\left(\boldsymbol{\tilde S}_l^{(1)}-\boldsymbol S_l^{(1)}\right)^2\right)\ge (2d+1)\varepsilon\right\}\ge \varepsilon \lfloor p/(2d+1)\rfloor \right)\\
&\hspace{0.5cm}\le \P\Bigg(\sum_{l=1}^{\lfloor p/(2d+1)\rfloor }\bigg(\ind\left\{\tr\left(\left(\boldsymbol{\tilde S}_l^{(1)}-\boldsymbol{S}_l^{(1)}\right)^2\right)\ge (2d+1)\varepsilon\right\}\\
&\hspace{3cm}-\E \ind\left\{\tr\left(\left(\boldsymbol{\tilde S}_l^{(1)}-\boldsymbol{S}_l^{(1)}\right)^2\right)\ge (2d+1)\varepsilon\right\} \bigg)\\
&\hspace{5cm}\ge \varepsilon \lfloor p/(2d+1)\rfloor-\frac{2}{\varepsilon n^2}\sum_{i=1}^p\sum_{k=1}^n (1-\E X_{ik}^2) \Bigg)\\
&\hspace{0.5cm}\le \P\Bigg(\sum_{l=1}^{\lfloor p/(2d+1)\rfloor }\bigg(\ind\left\{\tr\left(\left(\boldsymbol{\tilde S}_l^{(1)}-\boldsymbol S_l^{(1)}\right)^2\right)\ge (2d+1)\varepsilon\right\}\\
&\hspace{3cm}-\E \ind\left\{\tr\left(\left(\boldsymbol{\tilde S}_l^{(1)}-\boldsymbol{S}_l^{(1)}\right)^2\right)\ge (2d+1)\varepsilon\right\} \bigg)\ge \frac{\varepsilon}{2} \lfloor p/(2d+1)\rfloor\Bigg)\\
&\hspace{0.5cm}\le \exp\left(-\frac{\varepsilon^2p}{4(2d+1)}\right),
\end{align*}
and accordingly,
\begin{align*}
&\P\left(\sum_{l=1}^{\lfloor (p-d)/(2d+1)\rfloor }\ind\left\{\tr\left(\left(\boldsymbol{\tilde S}_l^{(2)}-\boldsymbol{S}_l^{(2)}\right)^2\right)\ge (2d+1)\varepsilon\right\}\ge \varepsilon \lfloor p/(2d+1)\rfloor \right)\\
&\hspace{8cm}\le \exp\left(-\frac{\varepsilon^2p}{4(2d+1)}\right).
\end{align*}
Combining Corollary \ref{corollary: A41} and Theorem \ref{theorem: A43} yields for $p$ sufficiently large,
\begin{align*}
d_L\Big(&F^{\boldsymbol S},F^{\boldsymbol{\tilde S}}\Big)\\
&\le \frac{4d}{p}+\frac{4d+2}{p}\sum_{l=1}^{\lfloor p/(2d+1)\rfloor }\ind\left\{\tr\left(\left(\boldsymbol{\tilde S}_l^{\boldsymbol A}-\boldsymbol S_l^{\boldsymbol{A}}\right)^2\right)\ge (2d+1)\varepsilon\right\}\\
&\hspace{0.5cm}+\frac{4d+2}{p}\sum_{l=1}^{\lfloor (p-d)/(2d+1)\rfloor }\ind\left\{\tr\left(\left(\boldsymbol{\tilde S}_l^{\boldsymbol B}-\boldsymbol{S}_l^{\boldsymbol B}\right)^2\right)\ge (2d+1)\varepsilon\right\}\\
&\hspace{0.5cm}+\Bigg[\frac{1}{p} \sum_{l=1}^{\lfloor p/(2d+1)\rfloor }\ind\left\{\tr\left(\left(\boldsymbol{\tilde S}_l^{\boldsymbol A}-\boldsymbol{S}_l^{\boldsymbol A}\right)^2\right)< (2d+1)\varepsilon\right\}\tr\left(\left(\boldsymbol{\tilde S}_l^{\boldsymbol A}-\boldsymbol S_l^{\boldsymbol{A}}\right)^2\right)\\
&\hspace{0.5cm}+\frac{1}{p} \sum_{l=1}^{\lfloor (p-d)/(2d+1)\rfloor }\ind\left\{\tr\left(\left(\boldsymbol{\tilde S}_l^{\boldsymbol B}-\boldsymbol{S}_l^{\boldsymbol B}\right)^2\right)< (2d+1)\varepsilon\right\}\\
&\hspace{6cm}\times\tr\left(\left(\boldsymbol{\tilde S}_l^{\boldsymbol B}-\boldsymbol S_l^{\boldsymbol{B}}\right)^2\right)\Bigg]^{1/3}\\
&\le 3\varepsilon + \sqrt[3]{3\varepsilon}
\end{align*}
with probability not larger than
\begin{align}
2\exp\left(-\frac{\varepsilon^2p}{4(2d+1)}\right) \label{eq: block pro}.
\end{align}
Note that the first term in the bound on $d_L(F^{\boldsymbol S},F^{\boldsymbol{\tilde S}})$ occurs by removing the rows and columns from $\boldsymbol S$ and $\boldsymbol 
{\tilde S}$ which are not covered by the block structures. The second and third term treat the blocks which are removed from $\boldsymbol S$ and $\boldsymbol 
{\tilde S}$  for irregularity. Finally, the last term bounds the L\'evy distance between the spectral measures of the reduced matrices. The terms \eqref{eq: block pro} are summable since
\begin{align*}
n\le \frac{p}{(\log p)^2}.
\end{align*}
As a consequence,
\begin{align*}
d_L\left(F^{\boldsymbol{S}},F^{\boldsymbol{\tilde S}}\right)\rightarrow 0
\end{align*}
almost surely as $p\to\infty$.
As before, redefine $\boldsymbol{\tilde S}$ by $\boldsymbol S$. It remains to rescale the diagonal entries of $\boldsymbol S$. Therefore, let $\boldsymbol{\tilde S}\in\R^{p\times p}$ have the same off-diagonal entries as $\boldsymbol S$ and 
$$\tilde S_{ii}=\frac{1}{n}\sum_{k=1}^n\tilde\sigma_{ik}^{-2}X_{ik}^2,~i=1,...,p.$$
Here, we may use similar arguments as for the rescaling of the off-diagonal entries but we choose $\alpha=1$ instead of $\alpha=2$ in Theorem \ref{theorem: A38}. By the Lidskii-Wielandt perturbation bound (1.2) in \cite{Li1999}, we have
\begin{align*}
\frac{1}{p}\sum_{i=1}^p\left|\lambda_i(\boldsymbol{\tilde S})-\lambda_i(\boldsymbol{S})\right|\le \frac{1}{p}\Arrowvert \boldsymbol{S}-\boldsymbol{\tilde S} \Arrowvert_{S_1}=\frac{1}{np}\sum_{i=1}^p\sum_{k=1}^n\left(\tilde\sigma_{ik}^{-2}-1\right)X_{ik}^2.
\end{align*}
Furthermore, for each $i=1,...,p$ holds
\begin{align*}
\sum_{k=1}^n\left(\tilde\sigma_{ik}^{-2}-1\right)\E X_{ik}^2\le \sum_{i=1}^p1- \E X_{ik}^2,
\end{align*}
and therefore by Markov's and Hoeffding's inequality together with \eqref{eq: mom} for $p$ sufficiently large,
\begin{align*}
&\P\left(\sum_{i=1}^p\ind\left\{\sum_{k=1}^n\left(\sigma_{ik}^{-2}-1\right)X_{ik}^2\ge \varepsilon n\right\}\ge \varepsilon p\right)\\
&\hspace{4cm}\le \P\Bigg(\sum_{i=1}^p\Bigg[\ind\left\{\sum_{k=1}^n\left(\sigma_{ik}^{-2}-1\right)X_{ik}^2\ge \varepsilon n\right\}\\
&\hspace{6cm}-\E\ind\left\{\sum_{k=1}^n\left(\sigma_{ik}^{-2}-1\right)X_{ik}^2\ge \varepsilon n\right\}\Bigg]\ge \frac{\varepsilon p}{2} \Bigg)\\
&\hspace{4cm}\le \exp\left(-\frac{\varepsilon^2p}{2}\right).
\end{align*}
Now, Theorem \ref{theorem: A38} and Theorem \ref{theorem: A43} yield
\begin{align*}
d_L(F^{\boldsymbol S},F^{\boldsymbol{\tilde S}})\le 2\varepsilon+\sqrt{\varepsilon}
\end{align*}
with probability
\begin{align*}
\exp\left(-\frac{\varepsilon^2p}{2}\right).
\end{align*}
Again, by the Borel-Cantelli lemma 
\begin{align*}
 d_L(F^S,F^{\tilde S})\longrightarrow 0
\end{align*} 
almost surely as $p\to\infty$.\\
Subsequently, we may assume that the matrix $\boldsymbol X$ has the following properties:
\begin{enumerate}
\item All entries $X_{ik}$ are centered.
\item All but $o(pn)$ entries are standardized and if an entry $X_{ik}$ is not standardized then $\E X_{ik}^2\le \eta_n$.
\item All entries of $\boldsymbol X$ are bounded by $\sqrt{\eta_nn}$, where $\eta_n\downarrow 0$ with $\eta_n\ge \frac{1}{\log n}$. 
\end{enumerate}
Finally, we replace the non-standardized entries of $\boldsymbol X$ by Rademacher variables. First, define $\boldsymbol{\tilde X}=(X_{ik}\ind\{\E X_{ik}^2=1\})_{ik}$. By an analogous line of reasoning as in the rescaling step follows
$$d_L\left(F^{\boldsymbol S},F^{\boldsymbol{\tilde S}}\right)\longrightarrow 0 \ \ \text{a.s. as $p\to \infty$},$$
where $$\boldsymbol{\tilde S}=\left(\frac{1}{n}\boldsymbol{\tilde X}\boldsymbol{\tilde X}'\right)\circ \boldsymbol 1_d.$$
Now, let $\boldsymbol{\hat X}\in\R^{p\times n}$ have the entries $\hat X_{ik}=X_{ik}\ind\{\E X_{ik}^2=1\}+\varepsilon_{ik}\ind\{\E X_{ik}^2<1\}$, where $\varepsilon_{ik},~i=1,\dots,p,~k=1,\dots,n,$ are independent Rademacher variables and independent of $\boldsymbol{X}$. Moreover, define 
\begin{align*}
\boldsymbol{\hat S}=\left(\frac{1}{n}\boldsymbol{\hat X}\boldsymbol{\hat X}'\right)\circ\boldsymbol1_d
\end{align*}
Again by Corollary \ref{corollary: A41},
\begin{align}
&d_L^3\left(F^{\boldsymbol{\hat S}},F^{\boldsymbol{\tilde S}}\right)\\
&\hspace{0.5cm}\le \frac{1}{d}\tr\left(\left(\boldsymbol{\hat S}-\boldsymbol{\tilde S}\right)^2\right)\\
&\hspace{0.5cm}= \frac{1}{pn^2}\sum_{i,j=1}^p\left(\sum_{k=1}^n \hat X_{ik}\hat X_{jk}-\tilde X_{ik}\tilde X_{jk} \right)^2\ind\{|i-j|\le d\}\notag\\
&\hspace{0.5cm}\le \frac{2}{pn^2}\sum_{i,j=1}^p\left(\sum_{k=1}^n \left(\hat X_{ik}-\tilde X_{ik}\right)\tilde X_{jk} \right)^2\ind\{|i-j|\le d\}\notag\\
&\hspace{4cm}+\left(\sum_{k=1}^n \left(\hat X_{jk}-\tilde X_{jk}\right)  \hat X_{ik}\right)^2\ind\{|i-j|\le d\}\notag\\
&\hspace{0.5cm}= \frac{2}{pn^2}\sum_{i=1}^p\left(\sum_{k=1}^n \left(\hat X_{ik}-\tilde X_{ik}\right)\hat X_{ik} \right)^2\label{eq: replace1}\\
&\hspace{1.5cm}+ \frac{2}{pn^2}\sum_{\substack{i,j=1\\i\neq j}}^p\left(\sum_{k=1}^n \left(\hat X_{ik}-\tilde X_{ik}\right)\tilde X_{jk} \right)^2\ind\{|i-j|\le d\}\label{eq: replace2}\\
&\hspace{4cm}+\left(\sum_{k=1}^n \left(\hat X_{jk}-\tilde X_{jk}\right)  \hat X_{ik}\right)^2\ind\{|i-j|\le d\}\label{eq: replace3}
\end{align}
For line \eqref{eq: replace1} we have
\begin{align*}
\eqref{eq: replace1}=\frac{2}{pn^2}\sum_{i=1}^p\left(\sum_{k=1}^n \ind\{\E X_{ik}^2<1\} \right)^2\le \frac{2}{pn}\sum_{i=1}^p\sum_{k=1}^n \ind\{\E X_{ik}^2<1\}\rightarrow 0.
\end{align*}
The terms $\eqref{eq: replace2}$ and $\eqref{eq: replace3}$ are handled the same way. Therefore, we just consider $\eqref{eq: replace2}$. Rewrite
\begin{align*}
\eqref{eq: replace2}&=\frac{2}{pn^2}\sum_{\substack{i,j=1\\i\neq j}}^p\sum_{k=1}^n\left(\hat X_{ik}-\tilde X_{ik}\right)^2 \E \tilde X_{jk}^2\ind\{|i-j|\le d\}\\
&\hspace{0.5cm}+\frac{2}{pn^2}\sum_{\substack{i,j=1\\i\neq j}}^p\sum_{k=1}^n\left(\hat X_{ik}-\tilde X_{ik}\right)^2 \left(\tilde X_{jk}^2-\E \tilde X_{jk}^2\right)\ind\{|i-j|\le d\}\\
&\hspace{0.5cm}+\frac{2}{pn^2}\sum_{\substack{i,j=1\\i\neq j}}^p\sum_{\substack{k_1,k_2=1\\k_1\neq k_2}}^n\left(\hat X_{ik_1}-\tilde X_{ik_1}\right)\tilde X_{jk_1}\left(\hat X_{ik_2}-\tilde X_{ik_2}\right)\tilde X_{jk_2}\ind\{|i-j|\le d\}
\end{align*}
Denote the first term by $I_1$, the second by $I_2$, and the third by $I_3$. $I_1$ vanishes asymptotically since
\begin{align*}
\frac{2}{pn^2}\sum_{\substack{i,j=1\\i\neq j}}^p\sum_{k=1}^n\left(\hat X_{ik}-\tilde X_{ik}\right)^2 \E \tilde X_{jk}^2\ind\{|i-j|\le d\}\le\frac{4d}{pn^2}\sum_{i=1}^p\sum_{k=1}^n\ind\{\E X_{ik}^2<1\}\rightarrow0.
\end{align*}
Let $Y_{ik}:=\tilde X_{jk}^2-\E \tilde X_{jk}^2$. As in inequality \eqref{eq: exp Y} we bound
\begin{align}
\E Y_{ik}^2\le 2\eta_n^{m-1}n^{m-1}.
\end{align}
Then we obtain for $I_2$,
\begin{align*}
\E I_2^4&\le \frac{2^8d^4}{p^4n^8}\sum_{j_1,j_2,j_3,j_4=1}^p\sum_{k_1,k_2,k_3,k_4=1}^n|\E Y_{j_1k_1}Y_{j_2k_2}Y_{j_3k_3}Y_{j_4k_4}|\\
&=\frac{2^8d^4}{p^4n^8}\sum_{j=1}^p\sum_{k=1}^n\E Y_{jk}^4+\frac{2^83d^4}{p^4n^8}\sum_{j=1}^p\sum_{\substack{k_1,k_2=1\\k_1\neq k_2}}^n \E Y_{jk_1}^2\E Y_{jk_2}^2\\
&\hspace{0.5cm}+\frac{2^83d^4}{p^4n^8}\sum_{\substack{j_1,j_2=1\\ j_1\neq j_2}}^p\sum_{k=1}^n\E Y_{j_1k}^2\E Y_{j_2k}^2+\frac{2^83d^4}{p^4n^8}\sum_{\substack{j_1,j_2=1\\ j_1\neq j_2}}^p\sum_{\substack{k_1,k_2=1\\ k_1\neq k_2}}^n\E Y_{j_1k_1}^2\E Y_{j_2k_2}^2\\
&\le  \frac{2^{9}d^4\eta_n^3}{p^3n^4}+\frac{2^{9}3d^4\eta_n^2}{p^3n^4}+\frac{2^83d^4\eta_n^2}{p^2n^5}+\frac{2^83d^4\eta_n^2}{p^2n^4},
\end{align*}
The last line is summable over $p$. Thus, by the Borel-Cantelli lemma, $I_2\to 0$ almost surely as $p\to\infty$. Now, consider $I_3$ and note that $\boldsymbol{\hat X}-\boldsymbol{\tilde X}$ and $\boldsymbol{\tilde X}$ are independent. Again, we evaluate the fourth moment
\begin{align*}
\E I_3^4&\le \frac{2^8d^4}{p^4n^8}\sum_{j=1}^p\sum_{\substack{k_1,\dots,k_8=1\\ k_{2l-1}\neq k_{2l}}}^n\left|\E \tilde X_{jk_1} \tilde X_{jk_2} \tilde X_{jk_3} \tilde X_{jk_4} \tilde X_{jk_5} \tilde X_{jk_6} \tilde X_{jk_7} \tilde X_{jk_8} \right|\\
&\hspace{0.5cm}+\frac{2^83d^4}{p^4n^8}\sum_{\substack{j_1,j_2=1 \\ j_1\neq j_2}}^p\sum_{\substack{k_1,\dots,k_8=1\\ k_{2l-1}\neq k_{2l}}}^n\left|\E \tilde X_{j_1k_1} \tilde X_{j_1k_2} \tilde X_{j_1k_3} \tilde X_{j_1k_4}\right| \left|\E\tilde X_{j_2k_5} \tilde X_{j_2k_6} \tilde X_{j_2k_7} \tilde X_{j_2k_8} \right|\\
&\le \frac{2^8d^4}{p^4n^8}\sum_{j=1}^p\sum_{\substack{k_1,\dots,k_8=1}}^n\left|\E \tilde X_{jk_1} \tilde X_{jk_2} \tilde X_{jk_3} \tilde X_{jk_4} \tilde X_{jk_5} \tilde X_{jk_6} \tilde X_{jk_7} \tilde X_{jk_8} \right|\\
&\hspace{0.5cm}+\frac{2^{10}3d^4}{p^4n^8}\sum_{\substack{j_1,j_2=1}}^p\sum_{\substack{k_1,k_2,k_3,k_4=1}}^n\E \tilde X_{j_1k_1}^2 \E\tilde X_{j_1k_2} ^2\E\tilde X_{j_2k_3}^2 \E\tilde X_{j_3k_4}^2\\
&\le  \frac{2^84140d^4}{p^3n^4}\left(1+\eta_n+\eta_n^2+\eta_n^3\right)+ \frac{2^{10}3d^4}{p^2n^4},
\end{align*}
where $4140$ is the $8$-th Bell number and gives the number of partitions of $\{1,\dots,8\}$. As for $I_2$, we obtain $I_3\rightarrow 0$ almost surely as $p\to\infty.$\\
In what follows, we may assume that the entries of $X$ are centered, standardized random variables bounded by $\eta_n\sqrt{n}$ for some decreasing sequence $(\eta_n)$ converging to 0.

\subsection{Almost sure convergence of moments}
We use the method of moments to prove the almost sure weak convergence of the sequence $F^{\boldsymbol S}$.  First we prove the convergence of the expected moments of $F^{\boldsymbol S}$. Let $l\in\N$ and $i_{l+1}:=i_1$, and define
$$m_{p,l}:= \int x^l \d F^{\boldsymbol S}(x).$$
Then, we conclude
\begin{align}
\E m_{p,l}&=\frac{1}{p}\E\tr\left(\boldsymbol S^l\right)\notag\\
&=\frac{1}{pn^l}\sum_{\substack{i_1,...,i_{l}=1\\ |i_j-i_{j+1}|\le d, \\j=1,...,l}}^p\sum_{k_1,...,k_{l}=1}^n\E\left(\prod_{j=1,...,l}X_{i_jk_j}X_{i_{j+1}k_{j}}\right)\label{eq: mom null}
\end{align} 
For a multi-index $(i_1,k_1,i_2,k_2,\dots,i_l,k_l,i_1)$, let $G=(V,E)$ be the graph with vertex set $V=\{i_1,\dots,i_l\}+\{k_1,\dots,k_l\}$, where the vertices $i_1,\dots,i_l$ are supposed to lie on the $I$-line and $k_1,\dots,k_l$ on the $K$-line, and edge set $$E=\{\{i_1,k_1\},\{k_1,i_2\},\dots,\{i_l,k_l\},\{k_l,i_1\}\}.$$ First note that
\begin{align*}
\E\left(\prod_{j=1,...,l}X_{i_jk_j}X_{i_{j+1}k_{j}}\right)=0
\end{align*}
if the walk $i_1,k_1,i_2,k_2,\dots,i_l,k_l,i_1$ does not cross each edge $e\in E$ at least twice. For $|E|\le l$, we have
\begin{align*}
\E\left(\prod_{j=1,...,l}X_{i_jk_j}X_{i_{j+1}k_{j}}\right)\le \eta_n^{2l-2|E|} n^{l-|E|},
\end{align*}
where equality holds for $|E|=l$. Since $G$ is connected, we conclude $ |V|-1\le|E|\le l$. This implies that only those indices $(i_1,\dots,k_l,i_1)$ contribute asymptotical to the sum \eqref{eq: mom null} for which $|V|-1=|E|=l$, and therefore the corresponding graphs $G$ need  to be trees. Hence, by Section \ref{section: combinatorial} it remains to consider the sum over canonical walks $i_1,\dots,k_l,i_1$ of $d$-banded ordered trees in $\mathcal{B}_{p,n,d,l+1}$. We conclude by Lemma \ref{lemma: banded trees},
\begin{align*}
&\lim_{p\to\infty}\E m_{p,l}\\
&\hspace{0.3cm}=\lim_{p\to\infty}\frac{1}{pn^l}|\mathcal{B}_{p,n,d,l+1}|\\
&\hspace{0.3cm}=\lim_{p\to\infty}\frac{1}{n^l}\sum_{G(i^c,k^c)}n^{l+1-|\{i^c_1,\dots,i^c_l\}|}\prod_{k^\ast}F(\deg(k^\ast)d,\deg(k^\ast),2d)\\
&\hspace{0.3cm}=\lim_{p\to\infty}\sum_{G(i^c,k^c)}n^{1-|\{i^c_1,\dots,i^c_l\}|}\prod_{k^\ast}\sum_{j=0}^{\deg(k^\ast)}(-1)^j\binom{\deg(k^\ast)}{j}\binom{\deg(k^\ast)d-2jd-1}{\deg(k^\ast)-1}\\
&\hspace{0.3cm}=\sum_{G(i^c,k^c)}\prod_{k^\ast}\sum_{j=0}^{\deg(k^\ast)}\ind\{k^\ast>2j\}(-1)^j(y(\deg(k^\ast-2j)))^{\deg(k^\ast)-1}\frac{\deg(k^\ast)}{j!(\deg(k^\ast)-j)!},
\end{align*}
where the outer sum runs over all canonical ordered trees $G(i^c,k^c)$ and the product runs over all vertices $k^\ast\in\{k_1^c,\dots,k_l^c\}$. Note that the cardinality of $\{k_1^c,\dots,k_l^c\}$ depends on the underlying canonical ordered tree and is given by $\max_{s=1,\dots,l}k_s^c$.\\
Lastly, we use once again the lemma of Borel-Cantelli to prove that $m_{p,l}-\E m_{p,l}\to 0$ almost surely as $p\to\infty$. Therefore, we evaluate the fourth moment of $m_{p,l}-\E m_{p,l}$. We follow the line of reasoning in \cite{Bai2010} on page 30 and 31. First, rewrite
\begin{align*}
\E \left(m_{p,l}-\E m_{p,l}\right)^4=p^{-4}n^{4l}\sum_{(i_j,k_j),~j=1,\dots,4}\E\prod_{j=1}^4\left(X[i_j,k_j]-\E X[i_j,k_j]\right),
\end{align*} 
where for any $j=1,2,3,4$, we denote $$k_j:=(k_{1,j},\dots,k_{1,j})\in [n]^l \ \ \text{and} \ \ i_j:=(i_{1,j},\dots,i_{l,j})\in [p]^l$$ such that $|i_{s,j}-i_{s+1,j}|\le d$ for $s=1,\dots,l$ with $i_{l+1,j}:=i_{1,j}$, and 
\begin{align*}
X[i_j,k_j]=\prod_{s=1}^l X_{i_{s,j}k_{s,j}}X_{i_{s+1,j}k_{s,j}}.
\end{align*}
Again, we assume the indices $i_{s,j}$ to lie on the $I$-line and $k_{s,j}$ on the $K$-line. Then for fixed $(i_j,k_j)$, $j=1,\dots,4$, define the graphs $G_j$ with vertex sets $$V_j:=\{i_{1,j},\dots,i_{l,j}\}+\{k_{1,j},\dots,k_{l,j}\}$$ and edge sets $$E_j:=\{\{i_{1,j},k_{1,j}\},\{i_{2,j},k_{1,j}\},\dots,\{i_{l,j},k_{l,j}\}\{i_{1,j},k_{l,j}\}\},$$ and $G$ with vertex set $V:=V_1\cup V_2\cup V_3\cup V_4$ and edge set $E:=E_1\cup E_2\cup E_3\cup E_4$. Now observe that
$$\E\prod_{j=1}^4\left(X[i_j,k_j]-\E X[i_j,k_j]\right)=0$$
if one of the graphs $G_j$ has no common edge with any of the other three, or if one edge $e\in E$ occurs only once in the sequence 
\begin{align*}
a:=&\{i_{1,1},k_{1,1}\},\{i_{2,1},k_{1,1}\},\dots,\{i_{l,1},k_{l,1}\},\{i_{1,1},k_{l,1}\},\\
&\{i_{1,2},k_{1,2}\},\{i_{2,2},k_{1,2}\},\dots,\{i_{l,2},k_{l,2}\},\{i_{1,2},k_{l,2}\},\\
&\{i_{1,3},k_{1,3}\},\{i_{2,3},k_{1,3}\},\dots,\{i_{l,3},k_{l,3}\},\{i_{1,1},k_{l,3}\},\\
&\{i_{1,4},k_{1,4}\},\{i_{2,4},k_{1,4}\},\dots,\{i_{l,4},k_{l,4}\},\{i_{1,4},k_{l,4}\}.
\end{align*}
We conclude that $G$ consists of at most two connected components, and each edge of a connected component occurs twice in $a$. In particular, $|V|\le |E|+2$. Denote the edges in $E$ by $e_1,\dots,e_{|E|}$ and by $\nu_1,\dots,\nu_{|E|}$ the corresponding multiplicities of the edges in the sequence $a$. Then,
\begin{align*}
p^{-4}n^{-4l}\left|\E \prod_{j=1}^4\left(X[i_j,k_j]-\E X[i_j,k_j]\right)\right|&\le 16p^{-4}n^{-4l}\eta_n^{\nu_1+\dots+\nu_{|E|}-2|E|}n^{\frac{\nu_1+\dots+\nu_{|E|}-2|E|}{2}}\\
&=16p^{-4}n^{-|E|}\eta_n^{8l-2|E|}.
\end{align*}
The number of indices $(i_j,k_j)$, $j=1,\dots,4$ such that the graph $G$ has at most two connected components, $|E|=s$, $s=1,\dots,2l$, and $|V|\le s+2$ is bounded by $C_lp^2n^s$, where the constant $C_l>0$ does only depend on $l$ and $\sup_p d/n<\infty$, and may be chosen uniformly over all $s=1,\dots,2l$. Alltogether,
\begin{align*}
\E\left(m_{p,l}-\E m_{p,l}\right)^4&=p^{-4}n^{4l}\sum_{(i_j,k_j),~j=1,\dots,4}\E\prod_{j=1}^4\left(X[i_j,k_j]-\E X[i_j,k_j]\right)\\&
\le 16\sum_{s\le 2l}C_lp^2n^sp^{-4}n^{-s}\eta_n^{8l-2|E|}\\ 
&\le 32lC_lp^{-2}.
\end{align*}
The last expression is summable over $p$, and therefore
$$m_{p,l}\longrightarrow m_l$$
almost surely as $p\to\infty$. 
\end{proof}

\bibliographystyle{imsart-nameyear}
\bibliography{reference}




%

%

%





\end{document}